\documentclass[reqno]{amsart}
\usepackage{amssymb, amsmath, amsthm, amscd, mathrsfs}

\usepackage{paralist}
\usepackage{cite}
\usepackage{bigints}
\usepackage{dsfont}
\usepackage{empheq}
\usepackage{amssymb}
\usepackage{cases}
\usepackage{enumitem}
\allowdisplaybreaks

\usepackage{caption}
\usepackage{booktabs}

\usepackage{float}
\usepackage[ruled,vlined,linesnumbered,noresetcount]{algorithm2e}

\usepackage{hyperref}

\theoremstyle{thmstyletwo}
\newtheorem{theorem}{Theorem}
\newtheorem{proposition}[theorem]{Proposition}
\newtheorem{teo}{Theorem}[section]
\newtheorem{lemma}{Lemma}[section]

\theoremstyle{definition}
\newtheorem{defi}[teo]{Definition}
\newtheorem{rmq}[teo]{Remark}

\numberwithin{equation}{section}

\def \d {\mathrm{d}}

\title[Logarithmic convexity and impulsive controllability] 
      {Logarithmic convexity and impulsive controllability for the 1-D heat equation with dynamic boundary conditions}

\author{S. E. Chorfi}
\author{G. El Guermai}
\author{L. Maniar}
\author{W. Zouhair}

\address{S. E. Chorfi, G. El Guermai, L. Maniar and W. Zouhair, Cadi Ayyad University, Faculty of Sciences Semlalia, LMDP, UMMISCO (IRD-UPMC), B.P. 2390, Marrakesh, Morocco}

\email{chorphi@gmail.com, ghita.el.guermai@gmail.com, maniar@uca.ma, walid.zouhair.fssm@gmail.com}

\makeatletter 
\@namedef{subjclassname@2020}{%
  \textup{2020} Mathematics Subject Classification}
\makeatother


\begin{document}
\begin{abstract}
In this paper, we prove a logarithmic convexity that reflects an observability estimate at a single point of time for 1-D heat equation with dynamic boundary conditions. Consequently, we establish the impulse approximate controllability for the impulsive heat equation with dynamic boundary conditions. Moreover, we obtain an explicit upper bound of the cost of impulse control. At the end, we give a constructive algorithm for computing the impulsive control of minimal $L^2$-norm. We also present some numerical tests to validate the theoretical results and show the efficiency of the designed algorithm.
\end{abstract}

\maketitle

\section{Introduction and main results}
Impulsive systems are of vital importance for most scientific fields; they can be found in many applications ranging from, among others, engineering, biology, population dynamics, economics, see e.g., \cite{articleNII,NIIexamples, MBEYR}. Many physical phenomena are modeled via evolution equations. Controls, impulses, and dynamic boundary conditions are often added to capture either a feedback or activity characterization.

In this paper, our interest is to investigate the impulse controlled heat equation with dynamic boundary conditions given by

\begin{empheq}[left = \empheqlbrace]{alignat=2} \label{1.1}
\begin{aligned}
&\partial_{t} \psi(x,t)-\partial_{xx} \psi(x,t)=0, && \qquad (x,t)\in(a,b) \times(0, T) \backslash\{\tau\},\\
&\psi(x, \tau)=\psi\left(x, \tau^{-}\right)+\mathds{1}_{\omega}(x) h(x), && \qquad x\in (a,b),\\
&\partial_{t}\psi(a,t) - \partial_{x} \psi(a,t)=0, && \qquad t\in(0, T) \backslash\{\tau\}, \\
&\partial_{t}\psi(b,t) + \partial_{x} \psi(b,t)=0, && \qquad t\in(0, T) \backslash\{\tau\}, \\
&\psi(a, \tau)=\psi\left(a, \tau^{-}\right),\\
&\psi(b, \tau)=\psi\left(b, \tau^{-}\right),\\
& \left(\psi(x,0),\psi(a, 0),\psi(b, 0)\right)=\left(\psi^0(x),c,d\right), && \qquad x\in(a,b),
\end{aligned}
\end{empheq}
where $\left(a,b\right) \subset \mathbb{R}$ is an open interval, $T > 0$ is the final time, $\tau \in (0,T)$ is an impulse time, $\left(\psi^{0},c,d\right)\in L^2\left(a,b\right)\times \mathbb R^2$ denotes the initial data, $\psi(\cdot,\tau^{-})$ denotes the left limit of the function $\psi$ at time $\tau$, $\omega\Subset (a,b)$ is a nonempty open subset, and $h\in L^2(\omega)$ is the impulse control.

The theory of impulsive differential equations was initiated by Milman and Mishkis in 1960 \cite{MVDM}. Afterward, many scientists contributed to the enrichment of this theory for more general evolution equations from various theoretical and numerical aspects. Many studies have been launched on this discipline and a large number of results has been reached. Controllability and observability are among the most important properties investigated within this theory, see for instance \cite{YT,Jose}. A system is controllable if we can drive the state from any initial condition to any desired target within a finite period of time either exactly or approximately. A system is observable if we can determine the state of the system based on some measured output data, see more in \cite{Bensoussan07, Zuazua07}.

Unlike the distributed controllability for impulsive systems that has been extensively studied in the literature, see for instance \cite{CDJUHLOC,ACCDGHL,GL,CDGHL,LHWZME,Lalvay} and the references therein, the problem of controllability with impulse controls has attracted less attention and not as many works are available in this area, we mention \cite{ABWZ,ka,pkdgwyx, QSGW}. In the later type, the control is a function acting only at one instant of time $\tau$. This is considerably different from the former type in which the control is acting in the whole time interval $(0, T)$.

Parabolic systems with dynamic boundary conditions, especially the one-dimensional systems, appear in many models and occupy a particular attention in recent literature. For instance, they model the distribution of heat in a bar of a given length \cite{KN'04}, the flow of heat for a solid in contact with a fluid \cite{La'32}, and the infiltration of water through a partially saturated porous medium (e.g., the rainfall through the soil) \cite{LS'85}. In the absence of impulses, the null controllability of the 1-D heat equation with dynamic boundary conditions has been recently studied in \cite{Kh'20}, Chapter 5, by using the moment method developed by Fattorini and Russell in \cite{FR'71}.

Let us briefly recall the derivation of a 1-D heat conduction model with dynamic boundary conditions. We consider the cooling of a uniform, isotropic, thin solid bar of cylindrical form, whose lateral surface is thermally insulated and its faces are placed at $x=a$ and $x=b$. Suppose that the two ends of the bar are placed in contact with a liquid and governed by an initial temperature $\psi^0$ at time $t=0$.

For simplicity and without loss of generality, we shall assume that the diffusivity is the unity constant $d=1$. The problem then is to find the temperature $\psi(x,t)$ of points in the bar or of the liquid at any intermediate time within a time horizon $[0,T]$. Using the law of conservation of energy and Fourier's law on heat conduction for the interior points of the bar, we obtain the classical heat equation
$$\partial_t \psi(x,t) -\partial_{xx} \psi(x,t)=0, \qquad (x,t) \in (a,b) \times (0,T).$$
The fact that the gain of heat by the liquid at the bar ends is equal to the loss of heat by the bar gives rise to the dynamic boundary conditions
\begin{align*}
\partial_{t}\psi(a,t) &= \partial_{x} \psi(a,t), \;\;\,\qquad t\in (0, T), \\
\partial_{t}\psi(b,t) &=- \partial_{x} \psi(b,t), \qquad t \in (0, T).
\end{align*}
The main aim of the model \eqref{1.1} is to control the heat distribution in the whole bar by acting only on an arbitrary internal part $\omega \Subset (a,b)$ at a single impulse time $\tau \in (0,T)$.

In the multi-dimensional case of a bounded domain $\Omega \subset \mathbb{R}^N$, $N\ge 2$, with a smooth boundary $\Gamma=\partial \Omega$, the dynamic boundary condition takes the form
\begin{equation}\label{dbcN}
\partial_t \psi_{\Gamma} -\delta \Delta_{\Gamma} \psi_{\Gamma} +\partial_\nu \psi=0,
\end{equation}
where $\psi_\Gamma=\psi_{|\Gamma}$ is the trace of $\psi$, $\Delta_{\Gamma}$ is the Laplace-Beltrami operator and $\partial_\nu \psi$ is the normal derivative with respect to the outward unit normal vector field $\nu$. The controllability and inverse problems for the non-impulsive heat equation with the dynamic boundary condition \eqref{dbcN} have been considered in the recent papers \cite{ACM'21',BCMO'20,MMS'17,KM'19,ACMO'20}. The impulse approximate controllability has been recently investigated in \cite{CGMZ'21}. In all these works, the presence of the diffusion on the boundary, i.e., $\delta>0$, has helped to overcome a technical difficulty in establishing observability estimates, see \cite[Remark 3.3]{MMS'17}. Such estimates for $\delta=0$ (absence of boundary diffusion) still open. Motivated by the aforementioned fact, we consider in the present work the one-dimensional case where no diffusion occur on the boundary. This makes our problem of particular interest and rather different from the previous works.

To prove the approximate controllability of system \eqref{1.1}, we establish an observability estimate supported at the final time $T$ and localized in the space subset $\omega$. This result will be done by employing a logarithmic convexity inequality based on a Carleman commutator approach, see e.g., \cite{ACM'21, ACM'21'', BCKDP, pkm, RBKDP}. This approach has been recently considered in \cite{pkm} and \cite{RBKDP} for heat equation with homogeneous Dirichlet and Neumann boundary conditions. In our setting, we deal with several new boundary terms that need to be absorbed. Moreover, no numerical results were presented in the previous works. Therefore, we will present an algorithm for the numerical computation of the impulse control of minimal $L^2$-norm. This will be achieved by adapting (to the impulsive case) the penalized Hilbert Uniqueness Method (HUM) and a Conjugate Gradient (CG) method. More precisely, the strategy that we follow is based on a formulation of the control problem under the form of a suitable convex quadratic optimization problem, which allows us to determine the minimum energy control. The interested reader can refer to the book \cite{GL'08} and the paper \cite{Bo'13}.

Next, we state our first main result which is an observability estimate at one instant of time. The proof is given in Section \ref{sec3}.
\begin{lemma}\label{lem1.1}
Let $\omega \Subset (a,b)$ be an open nonempty set. Let $\langle \cdot, \cdot\rangle$ denote the standard inner product of
$L^2(a,b) \times \mathbb{R}^{2}$ and $\|\cdot \|$ be its corresponding norm. Then the following logarithmic convexity estimate holds
\begin{equation}\label{1.2}
\|U(\cdot, T)\| \leq\left(\mu \mathrm{e}^{\frac{K}{T}}\|u(\cdot, T)\|_{L^{2}(\omega)}\right)^{\beta}\|U(\cdot, 0)\|^{1-\beta},
\end{equation}
where $\mu, K >0$, $\beta \in (0,1)$ and $U=\left(u(\cdot,\cdot),u(a,\cdot),u(b,\cdot)\right)$ is the solution of the following system
\begin{empheq}[left = \empheqlbrace]{alignat=2}\label{1.3}
\begin{aligned}
&\partial_{t} u(x,t)-\partial_{xx} u(x,t)=0, && \qquad (x,t)\in (a,b) \times(0, T) , \\
&\partial_{t}u(a,t) - \partial_{x}u(a,t) =0, && \qquad t\in (0, T), \\
&\partial_{t}u(b,t) + \partial_{x}u(b,t) =0, && \qquad t\in (0, T), \\
&\left(u(x,0),u(a, 0),u(b, 0)\right)=\left(u^0(x),c_{1},d_{1}\right), && \qquad x\in (a,b).
\end{aligned}
\end{empheq}
\end{lemma}
We will prove the estimate \eqref{1.2} keeping track of the explicit dependence of all the relevant constants with respect to different parameters. This will allow us to obtain an explicit upper bound for the cost of the impulsive approximate controllability of system \eqref{1.1}.

\begin{rmq} Some remarks are in order:
\begin{itemize}
    \item We emphasize that \eqref{1.2} is an observability inequality estimating the whole solution $U=\left(u(\cdot,\cdot),u(a,\cdot),u(b,\cdot)\right)$ of system \eqref{1.3} at terminal time $T$. This is done by only using one internal observation on the first component $u$, which is localized in the subset $\omega$.
    \item The estimate \eqref{1.2} is of independent interest. It allows one to prove an observability inequality for the non-impulsive system from any measurable set $E \subset (0,T]$ with a positive Lebesgue measure. This can be used to study the bang-bang property for the time optimal control problem \cite{pkdlwzc}.
\end{itemize}
\end{rmq}

The rest of the paper is organized as follows: in Section \ref{sec2}, we briefly recall some results on the wellposedness of the system. In Section \ref{sec3}, we present the strategy to obtain the observation estimate at one point of time for system \eqref{1.1}. Section \ref{sec4} is devoted to the impulse approximate controllability of the system. Finally, Section \ref{sec5} includes a constructive algorithm for computing the impulse control of minimal $L^2$-norm illustrated by numerical simulations.

\section{Wellposedness of the system} \label{sec2}
In this section, we recall some results that will be useful in the sequel. We will often use the following real Hilbert space $\mathbb{L}^2:=L^2(a,b)\times \mathbb{R}^{2}$, equipped with the inner product
\begin{align*}
\langle (u,c,d),(v,c_{1},d_{1})\rangle =\langle u,v\rangle_{L^2(a,b)} +c c_{1}+d d_{1}.
\end{align*}
We also consider the spaces
$$\mathbb{H}^k:=\left\{(u,u(a),u(b)) \colon u\in H^k\left(a,b\right)\right\} \text{ for } k=1,2,$$ equipped with the standard product norms.

System \eqref{1.3} can be written as the following abstract Cauchy problem
\begin{numcases}{\text{(ACP)}\quad\label{acp}}
\hspace{-0.1cm} \partial_t \mathbf{U}=\mathbf{A} \mathbf{U}, \quad 0<t \le T, \nonumber\\
\hspace{-0.1cm} \mathbf{U}(0)=\left(u^0, c_1,d_1\right), \nonumber
\end{numcases}
where $\mathbf{U}:=(u,u\left(a,\cdot\right),u\left(b,\cdot\right))$ and the linear operator $\mathbf{A} \colon D(\mathbf{A}) \subset \mathbb{L}^2 \longrightarrow \mathbb{L}^2$ is given by
\begin{equation*}
\mathbf{A}=\begin{pmatrix}
\partial_{xx} & 0 & 0 \\[3mm]
{\partial_{x}}_{|x=a} & 0 & 0\\[3mm]
{-\partial_{x}}_{|x=b} & 0 & 0
\end{pmatrix} , \qquad \qquad D(\mathbf{A})=\mathbb{H}^2. \label{E21}
\end{equation*}

We recall the following generation result that has been proven in \cite{Kh'20}, Proposition 5.2.1.
\begin{proposition}
The operator $\mathbf{A}$ is densely defined, self-adjoint and generates an analytic $C_{0}$-semigroup of contractions $\left(\mathrm{e}^{t \mathbf{A}}\right)_{t \geq 0}$ of angle $\dfrac{\pi}{2}$ on $\mathbb{L}^{2}$.
\end{proposition}
It follows that the solution map $t\mapsto \mathrm{e}^{t\mathbf{A}} \mathbf{U}_0$ is infinitely many times differentiable for $t>0$ and $\mathrm{e}^{t\mathbf{A}} \mathbf{U}_0 \in D\left(\mathbf{A}^m\right)$ for every $\mathbf{U}_0 \in \mathbb{L}^2$ and $m\in \mathbb{N}$.

On the other hand, the system \eqref{1.1} can be presented as the following impulsive Cauchy problem
\begin{equation}\label{ACP}
\text{(ACP)} \;\; \begin{cases}
\hspace{-0.1cm} \partial_t \Psi(t)=\mathbf{A} \Psi(t), \quad (0,T)\setminus\{\tau\}, \nonumber\\
\hspace{-0.1cm} \bigtriangleup \Psi(\tau) = (\mathds{1}_{\omega} h, 0,0) ,\nonumber\\
\hspace{-0.1cm} \Psi(0)=(\psi^0, c,d), \nonumber
\end{cases}
\end{equation}
where $\Psi:=(\psi,\psi\left(a,\cdot\right),\psi\left(b,\cdot\right))$ and $\bigtriangleup \Psi(\tau) :=\Psi\left(\cdot,\tau\right) - \Psi\left(\cdot,\tau^{-}\right)$.
For all $\Psi_{0}:=(\psi^0, c,d) \in \mathbb{L}^2$, the system (ACP) has a unique mild solution given by
$$
\Psi(t) = \mathrm{e}^{t\mathbf{A}} \Psi_{0} + \mathds{1}_{\{t\geq \tau \}}(t)\, \mathrm{e}^{(t-\tau)\mathbf{A}} (\mathds{1}_{\omega} h,0,0), \qquad t\in (0,T).
$$
\section{Logarithmic convexity estimate} \label{sec3}
In this section, we prove Lemma \ref{lem1.1}. We shall follow the strategy presented in \cite{pkm} in modified form. In our context, we need to collect and treat several new boundary terms arising from the dynamic boundary condition.

Following \cite{RBKDP}, we consider the weight function $\Phi: [a,b] \times (0,T)\rightarrow \mathbb{R}$ given by
\[
\Phi(x,t)=\frac{- s \left\lvert x-x_{0}\right\lvert^2}{4(T-t+\hbar )},
\]
where $x_0 \in \omega,\; \hbar > 0$ and $s \in (0,1)$. Let us set
\[
\varphi(x)=\frac{-\left\lvert x-x_{0}\right\lvert^2}{4}\quad \text{and}\quad \Upsilon(t)=T-t+\hbar,
\]
then
$$ \Phi (x,t)=\frac{s\varphi(x)}{\Upsilon(t)}, \qquad \forall (x,t) \in [a,b] \times (0,T).$$
The function $\varphi$ satisfies the following properties
\begin{itemize}
\item[(1)] $\varphi(x)+\left\lvert \varphi_x(x)\right\lvert^{2} = 0,\; \forall x\in [a,b]$,
\item[(2)] $\varphi_x(x)=-\dfrac{1}{2}(x-x_{0}), \;\forall x\in [a,b]$,
\item[(3)] $\varphi_{xx}(x) =-\dfrac{1}{2},\; \forall x\in [a,b]$.
\end{itemize}

\begin{proof}[Proof of Lemma \ref{lem1.1}]
We will divide the proof into several steps.

\noindent\textbf{Step 1.} Let $(u^0,c_1,d_1) \in \mathbb{L}^2 \setminus \{(0,0,0)\}$. Define
\begin{equation*}
F(x, t)=U(x, t)\,\mathrm{e}^{\Phi(x, t) / 2},
\end{equation*}
where $U(x,t) :=\begin{pmatrix}
u(x,t) \\ u(a,t)\\ u(b,t)
\end{pmatrix}$ is the solution of \eqref{1.3} and $F(x,t) :=\begin{pmatrix}
f(x,t) \\ f(a,t)\\ f(b,t)
\end{pmatrix}$. Define the operator $P$ as follows
\begin{equation*}
P F = \mathrm{e}^{\Phi / 2}
\begin{pmatrix}
\partial_{t}-\partial_{xx} & 0 & 0 \\[3mm]
{-\partial_{x}}_{|x=a} & \partial_{t} & 0\\[3mm]
{\partial_{x}}_{|x=b} & 0 & \partial_{t}
\end{pmatrix} \mathrm{e}^{-\Phi / 2} F.
\end{equation*}
Then
\begin{equation*}
P F =
\begin{pmatrix}
\partial_{t} f-\partial_{xx} f-\frac{1}{2}f\left(\partial_{t} \Phi+\frac{1}{2}|\partial_{x} \Phi|^{2}\right)+ \partial_{x} \Phi \partial_{x} f +\frac{1}{2}\partial_{xx} \Phi f \\[3mm]
\partial_{t} f(a,t)- \partial_{x} f(a,t) -\frac{1}{2} f(a,t)\partial_{t} \Phi(a,t)+\frac{1}{2} \partial_{x}\Phi(a,t) f(a,t) \\[3mm]
\partial_{t} f(b,t)+ \partial_{x} f(b,t)-\frac{1}{2} f(b,t)\partial_{t} \Phi(b,t)-\frac{1}{2} \partial_{x}\Phi(b,t) f(b,t)
\end{pmatrix}.
\end{equation*}
Let us define $P_{1}$ as follows
\begin{equation*}
P_{1} F =
\begin{pmatrix}
\partial_{xx} f+\frac{1}{2}f\left(\partial_{t} \Phi+\frac{1}{2}|\partial_{x} \Phi|^{2}\right)- \partial_{x} \Phi \partial_{x} f -\frac{1}{2}\partial_{xx} \Phi f \\[3mm]
 \partial_{x}f(a,t)  +\frac{1}{2} f(a,t)\partial_{t} \Phi(a,t)-\frac{1}{2} \partial_{x}\Phi(a,t) f(a,t) \\[3mm]
- \partial_{x} f(b,t) +\frac{1}{2} f(b,t)\partial_{t} \Phi(b,t)+\frac{1}{2} \partial_{x}\Phi(b,t) f(b,t)
\end{pmatrix}.
\end{equation*}
Since $U$ solves \eqref{1.3}, then $P F =0$. Thus
\begin{equation*}
\begin{pmatrix}
\partial_{t} f (x,t)\\
\partial_{t} f(a,t)\\
\partial_{t} f(b,t)
\end{pmatrix} = P_{1} F(x,t),\qquad \forall (x,t)\in (a,b)\times (0,T).
\end{equation*}
Let us compute the adjoint operator of $P_1.$ For any $G :=\begin{pmatrix}
g \\
g(a)\\
g(b)
\end{pmatrix} \in \mathbb{H}^1$,
\begin{align*}
\left\langle P_1 F , G \right\rangle &=\displaystyle\int_{a}^b \partial_{xx} f g \mathrm{d} x+\int_{a}^b \frac{1}{2} f g\left(\partial_{t} \Phi+\frac{1}{2}|\partial_{x} \Phi|^{2}\right) \mathrm{d}x- \int_{a}^b g \partial_{x} \Phi \partial_{x}f \mathrm{d} x\\
&-\displaystyle\frac{1}{2} \int_{a}^b  \partial_{xx} \Phi fg \mathrm{d}x +\partial_{x}f(a,t)g(a)+\frac{1}{2}f(a,t)g(a)\partial_{t}\Phi(a,t)\\
&-\frac{1}{2}\partial_{x}\Phi(a,t)f(a,t)g(a)-\partial_{x}f(b,t)g(b)+\frac{1}{2}f(b,t)g(b)\partial_{t}\Phi(b,t)\\
&+\frac{1}{2}\partial_{x}\Phi(b,t)f(b,t)g(b)\\
&= \partial_{x}f(b,t)g(b)-\partial_{x}f(a,t)g(a) - \int_{a}^b \partial_{x} f  \partial_{x} g \mathrm{d} x \\
&+\frac{1}{2}\int_{a}^b  f g\left(\partial_{t} \Phi+\frac{1}{2}|\partial_{x} \Phi|^{2}\right) \mathrm{d}x-  \partial_{x}\Phi(b,t) f(b,t) g(b)\\
&+\partial_{x}\Phi(a,t)f(a,t)g(a) +\int_{a}^b\partial_{xx}\Phi f g \mathrm{d}x+\int_{a}^b f\partial_{x}\Phi\partial_{x}g\mathrm{d}x\\
&-\displaystyle \frac{1}{2} \int_{a}^b \partial_{xx} \Phi f g \mathrm{d}x  +\partial_{x}f(a,t)g(a)+\frac{1}{2}f(a,t)g(a)\partial_{t}\Phi(a,t)\\
&-\frac{1}{2}\partial_{x}\Phi(a,t)f(a,t)g(a)-\partial_{x}f(b,t)g(b)+\frac{1}{2}f(b,t)g(b)\partial_{t}\Phi(b,t)\\
&+\frac{1}{2}\partial_{x}\Phi(b,t)f(b,t)g(b)\\
&= -f(b,t)\partial_{x}g(b)+f(a,t)\partial_{x}g(a)+\int_{a}^b f\partial_{xx}g\mathrm{d}x\\
&+\frac{1}{2}\int_{a}^b fg\left(\partial_{t}\Phi+\frac{1}{2}\lvert\partial_{x}\Phi\lvert^2\right)\mathrm{d}x-\frac{1}{2}\partial_{x}\Phi(b,t)f(b,t)g(b)\\
&+\frac{1}{2}\partial_{x}\Phi(a,t)f(a,t)g(a)+\frac{1}{2}\int_{a}^b\partial_{xx}\Phi fg\mathrm{d}x+\int_{a}^b f\partial_{x}\Phi\partial_{x}g\mathrm{d}x\\
&+\frac{1}{2}f(a,t)g(a)\partial_{t}\Phi(a,t)+\frac{1}{2}f(b,t)g(b)\partial_{t}\Phi(b,t)\\
&= \left\langle F, P_{1}^{*} G \right\rangle.
\end{align*}
Next, we introduce the following operator
\begin{equation*}
\mathcal{A} = \frac{P_1 - P_{1}^{*}}{2} = \begin{pmatrix}
-\partial_{x}\Phi\partial_{x} -\frac{1}{2}\partial_{xx}\Phi  & 0 & 0\\[3mm] 
0 & -\frac{1}{2} \partial_{x}\Phi(a,t) & 0 \\[3mm]
0 & 0 & \frac{1}{2} \partial_{x}\Phi(b,t)
\end{pmatrix},
\end{equation*}
which is antisymmetric on $\mathbb{H}^1$. Similarly, we define the following operator
\begin{equation*}
\mathcal{S} = \frac{P_1 + P_{1}^{*}}{2} = \begin{pmatrix}
\partial_{xx} + \eta & 0 & 0 \\[3mm]
{\partial_{x}}_{| x=a} & \frac{1}{2}\partial_{t}\Phi(a,t) & 0 \\[3mm]
-{\partial_{x}}_{| x=b} & 0 & \frac{1}{2}\partial_{t}\Phi(b,t)
\end{pmatrix},\\
\end{equation*}
that is symmetric on $\mathbb{H}^1$, where
\[
\eta = \frac{1}{2} \left( \partial_{t}\Phi + \frac{1}{2}\left\lvert\partial_{x} \Phi \right\lvert^{2}\right).\\
\]
Thus,
\begin{equation*}
\partial_{t} F = \mathcal{S} F + \mathcal{A} F.
\end{equation*}

\noindent\textbf{Step 2.} Multiplying the above equation by $F$, we obtain
\begin{equation*}
\frac{1}{2} \partial_{t} \|F\|^{2} - \left\langle \mathcal{S}F, F \right\rangle =0.
\end{equation*}
We define the frequency function by
\[
\mathcal{N}=\frac{\langle-\mathcal{S} F, F\rangle}{\|F\|^{2}}
.\]
Then 
\begin{equation*}
\frac{1}{2} \partial_{t} \|F\|^{2} + \mathcal{N} \|F\|^{2}  =0.
\end{equation*}
The derivative of $\mathcal{N}$ satisfies
\begin{equation*}
\begin{array}[c]{ll}
\dfrac{\d}{\d t} \mathcal{N} & \leq \dfrac{1}{\|F\|^{2}}\bigg( \left\langle-\left(\mathcal{S}^{\prime}+[\mathcal{S}, \mathcal{A}]\right) F, F\right\rangle- \partial_{x}f(b,t)\left[\mathcal{A}_{1}f(b,t)-\mathcal{A}_{3}f(b,t)\right] \nonumber\\
 &-\partial_{x}f(a,t)\left[\mathcal{A}_{2}f(a,t)-\mathcal{A}_{1}f(a,t)\right]-\partial_{x}\Phi(a,t)f(a,t)\left[\mathcal{S}_{1}f(a,t)-\mathcal{S}_{2}f(a,t)\right]\\
 &-\partial_{x}\Phi(b,t)f(b,t)\left[\mathcal{S}_{3}f(b,t)-\mathcal{S}_{1}f(b,t)\right] \bigg),
\end{array}
\end{equation*}
where
$ \mathcal{S}F =\left(\begin{array}{l} \mathcal{S}_{1} f \\  \mathcal{S}_{2} f \left( a, . \right)\\\mathcal{S}_{3} f \left( b, . \right)\end{array}\right)$, $ \mathcal{A}F =\left(\begin{array}{l} \mathcal{A}_{1} f \\  \mathcal{A}_{2} f\left(a,.\right)\\\mathcal{A}_{3} f\left(b,.\right)
\end{array}\right)$and $ [\mathcal{S}, \mathcal{A}] F = \mathcal{S}  \mathcal{A} F - \mathcal{A} \mathcal{S} F$.

\noindent Indeed,
\begin{align}\label{2.17}
\frac{\d}{\d t} \mathcal{N} &= \partial_{t}\left(\frac{\langle-\mathcal{S} F, F\rangle}{\|F\|^{2}} \right) \nonumber \\
&=\frac{1}{\|F\|^{4}}\left(\partial_{t}\left(\langle-\mathcal{S} F, F\rangle\right)\|F\|^{2}+\langle\mathcal{S} F, F\rangle \partial_{t}\|F\|^{2}\right)\nonumber \\
&=\frac{1}{\|F\|^{2}}\left[\left\langle-\mathcal{S}^{\prime} F, F\right\rangle-2\left\langle\mathcal{S} F, F^{\prime}\right\rangle\right]+\frac{2}{\|F\|^{4}}\langle\mathcal{S} F, F\rangle^{2}\nonumber \\
&=\frac{1}{\|F\|^{2}}\left[\left\langle-\mathcal{S}^{\prime} F, F\right\rangle-2\langle\mathcal{S} F, \mathcal{A} F\rangle\right]+\frac{2}{\|F\|^{4}}\left[-\|\mathcal{S} F\|^{2}\|F\|^{2}+\langle\mathcal{S} F, F\rangle^{2}\right]\nonumber\\
&\leq \frac{1}{\|F\|^{2}}\left[\left\langle-\mathcal{S}^{\prime} F, F\right\rangle-2\langle\mathcal{S} F, \mathcal{A} F\rangle\right].
\end{align}
Next, we calculate $2 \langle\mathcal{S} F, \mathcal{A}
F\rangle$,
\begin{align*}
\langle\mathcal{S} F, \mathcal{A} F\rangle &= \int_{a}^b \left( \partial_{xx} f + \eta f\right) \mathcal{A}_{1}f \mathrm{d} x + \partial_{x}f(a,t)\mathcal{A}_{2}f(a,t)+\frac{1}{2}\partial_{t}\Phi(a,t) f(a,t)\mathcal{A}_{2}f(a,t)\\
& - \partial_{x}f(b,t)\mathcal{A}_{3}f(b,t)+\frac{1}{2}\partial_{t}\Phi(b,t) f(b,t)\mathcal{A}_{3}f(b,t)\\
&=\displaystyle\partial_{x}f(b,t)\mathcal{A}_{1}f(b,t)-\partial_{x}f(a,t)\mathcal{A}_{1}f(a,t)-\int_{a}^b \partial_{x}f \partial_{x} \mathcal{A}_{1}f\mathrm{d}x\\
&+\int_{a}^b \eta f \mathcal{A}_{1}f\mathrm{d}x+\partial_{x}f(a,t)\mathcal{A}_{2}f(a,t)+\frac{1}{2}\partial_{t}\Phi(a,t) f(a,t) \mathcal{A}_{2}f(a,t)\\
&-\partial_{x}f(b,t)\mathcal{A}_{3}f(b,t)+\frac{1}{2}\partial_{t}\Phi(b,t) f(b,t)\mathcal{A}_{3}f(b,t)\\
&= \partial_{x}f(b,t)\mathcal{A}_{1}f(b,t)-\partial_{x}f(a,t)\mathcal{A}_{1}f(a,t)-f(b,t)\partial_{x}\mathcal{A}_{1}f(b,t)\\
&+ f(a,t)\partial_{x}\mathcal{A}_{1}f(a,t)+\int_{a}^b f\partial_{xx}\mathcal{A}_{1}f \mathrm{d}x+\int_{a}^b \eta f \mathcal{A}_{1}f\mathrm{d}x\\
&+\partial_{x}f(a,t)\mathcal{A}_{2}f(a,t)+\frac{1}{2}\partial_{t}\Phi(a,t) f(a,t) \mathcal{A}_{2}f(a,t)-\partial_{x}f(b,t)\mathcal{A}_{3}f(b,t)\\
&+\frac{1}{2}\partial_{t}\Phi(b,t) f(b,t)\mathcal{A}_{3}f(b,t)\\
& =\langle\mathcal{S}\mathcal{A} F, F\rangle+\partial_{x}f (b,t)\mathcal{A}_{1}f(b,t)-\partial_{x}f(a,t) \mathcal{A}_{1}f(a,t)\\
&+\partial_{x}f(a,t)\mathcal{A}_{2}f(a,t)-\partial_{x}f(b,t) \mathcal{A}_{3}f(b,t).
\end{align*}
Therefore, we obtain
\begin{align}\label{3.15}
\langle\mathcal{S} F, \mathcal{A} F\rangle& = \langle\mathcal{S}  \mathcal{A} F, F\rangle + \partial_{x}f(b,t)\left[\mathcal{A}_{1}f(b,t)-\mathcal{A}_{3}f(b,t)\right]\nonumber\\
&+\partial_{x}f(a,t)\left[\mathcal{A}_{2}f(a,t)-\mathcal{A}_{1}f(a,t)\right].
\end{align}
Similarly, we show that 
\begin{align}\label{3.16}
\langle\mathcal{S} F, \mathcal{A} F\rangle &= - \langle\mathcal{A} \mathcal{S}   F, F\rangle +\partial_{x}\Phi(a,t)f(a,t)\left[\mathcal{S}_{1}f(a,t)-\mathcal{S}_{2}f(a,t)\right]\nonumber\\
&+\partial_{x}\Phi(b,t)f(b,t)\left[\mathcal{S}_{3}f(b,t)-\mathcal{S}_{1}f(b,t)\right].
\end{align}
Combining \eqref{3.15} and \eqref{3.16} yields
\begin{align}\label{2.20}
 2 \langle\mathcal{S} F, \mathcal{A} F\rangle =& \langle\mathcal{S}  \mathcal{A} F, F\rangle-\langle\mathcal{A}  \mathcal{S} F, F\rangle   + \partial_{x}f(b,t)\left[\mathcal{A}_{1}f(b,t)-\mathcal{A}_{3}f(b,t)\right] \nonumber\\
 & \hspace{-0.7cm} +\partial_{x}f(a,t)\left[\mathcal{A}_{2}f(a,t)-\mathcal{A}_{1}f(a,t)\right]+\partial_{x}\Phi(a,t)f(a,t)\left[\mathcal{S}_{1}f(a,t)-\mathcal{S}_{2}f(a,t)\right]\\\nonumber
 &+\partial_{x}\Phi(b,t)f(b,t)\left[\mathcal{S}_{3}f(b,t)-\mathcal{S}_{1}f(b,t)\right].
 \end{align}
Then, \eqref{2.17} and \eqref{2.20} imply the desired formula.

\noindent\textbf{Step 3.} The following identity holds:
\begin{align}\label{9.1}
&\left\langle-\left(\mathcal{S}^{\prime}+[\mathcal{S}, \mathcal{A}]\right) F, F\right\rangle- \partial_{x}f(b,t)\left[\mathcal{A}_{1}f(b,t)-\mathcal{A}_{3}f(b,t)\right] \nonumber\\
 &-\partial_{x}f(a,t)\left[\mathcal{A}_{2}f(a,t)-\mathcal{A}_{1}f(a,t)\right]-\partial_{x}\Phi(a,t)f(a,t)\left[\mathcal{S}_{1}f(a,t)-\mathcal{S}_{2}f(a,t)\right]\nonumber\\
 &-\partial_{x}\Phi(b,t)f(b,t)\left[\mathcal{S}_{3}f(b,t)-\mathcal{S}_{1}f(b,t)\right]\nonumber\\
&= \frac{-s(2-s)^{2}}{4\Upsilon^{3}}\int_{a}^{b}\varphi\lvert f \rvert^{2} \mathrm{d}x + \frac{s}{\Upsilon}\int_{a}^{b}\lvert \partial_{x}f \rvert^{2} \mathrm{d}x + \frac{s}{2\Upsilon}(2\partial_{x}\varphi(b)-1) \partial_{x}f(b,t)f(b,t)\\
&+ \frac{s}{2\Upsilon}(2\partial_{x}\varphi(a)+1) \partial_{x}f(a,t)f(a,t) - \frac{s}{\Upsilon^{3}} \varphi(a)\lvert f(a,t) \rvert^{2} - \frac{s}{\Upsilon^{3}} \varphi(b)\lvert f(b,t) \rvert^{2}\nonumber\\
&+\frac{s}{\Upsilon} \partial_{x}\varphi(b)\lvert\partial_{x} f(b,t) \rvert^{2}
- \frac{s}{\Upsilon}\partial_{x}\varphi(a)\lvert\partial_{x} f(a,t) \rvert^{2} + \frac{s^{3}}{4\Upsilon^{3}}\varphi(a)\partial_{x}\varphi(a)\lvert f(a,t) \rvert^{2}\nonumber\\
&- \frac{s^{3}}{4\Upsilon^{3}}\varphi(b)\partial_{x}\varphi(b)\lvert f(b,t) \rvert^{2} \nonumber.
\end{align}

\noindent Indeed,
\begin{align*}
\mathcal{S}\mathcal{A}F =  \begin{pmatrix}
\partial_{xx}\left(-\partial_{x} \Phi \partial_{x} f - \frac{1}{2}\partial_{xx}\Phi f \right)-\eta\partial_{x} \Phi\partial_{x} f -\frac{1}{2}\eta\partial_{xx}\Phi f \\[3mm]
{\partial_{x}}_{|x=a}\left(-\partial_{x}\Phi\partial_{x}f-\frac{1}{2}\partial_{xx}\Phi f\right)-\frac{1}{4}\partial_{t}\Phi(a,t)\partial_{x}\Phi(a,t)f(a,t)\\[3mm]
-{\partial_{x}}_{|x=b}\left(-\partial_{x}\Phi\partial_{x}f-\frac{1}{2}\partial_{xx}\Phi f\right)+\frac{1}{4}\partial_{t}\Phi(b,t)\partial_{x}\Phi(b,t)f(b,t)
\end{pmatrix},
\end{align*}
and
\begin{align*}
\mathcal{A}\mathcal{S}F =  \begin{pmatrix}
-\partial_{x}\Phi\partial_{x}\left(\partial_{xx}f\right)-\partial_{x}\Phi\partial_{x}\eta f-\eta \partial_{x}\Phi\partial_{x}f-\frac{1}{2}\partial_{xx}\Phi\partial_{xx}f-\frac{1}{2}\eta \partial_{xx}\Phi f \\[3mm]
-\frac{1}{2}\partial_{x}\Phi(a,t)\partial_{x}f(a,t)-\frac{1}{4}\partial_{t}\Phi(a,t)\partial_{x}\Phi (a,t)f(a,t)\\[3mm]
-\frac{1}{2}\partial_{x}\Phi(b,t)\partial_{x}f(b,t)+\frac{1}{4}\partial_{t}\Phi(b,t)\partial_{x}\Phi (b,t)f(b,t)
\end{pmatrix}.
\end{align*}
Also,
\begin{align*}
        \mathcal{S}^{\prime}F &=\left(\mathcal{S}F\right)^{\prime} - \mathcal{S} F^{\prime} \notag\\
&= \begin{pmatrix}
\partial_{xx}\left(\partial_{t}f\right)+\partial_{t}\eta f+\eta \partial_{t} f \\[3mm]
\partial_{t}\partial_{x}f(a,t) +\frac{1}{2}\partial_{tt} \Phi(a,t) f(a,t)+\frac{1}{2}\partial_{t}\Phi \partial_{t}f(a,t)\\[3mm]
-\partial_{t}\partial_{x}f(b,t) +\frac{1}{2}\partial_{tt} \Phi(b,t) f(b,t)+\frac{1}{2}\partial_{t}\Phi(b,t) \partial_{t}f(b,t)
\end{pmatrix} - \begin{pmatrix}
\partial_{xx}\partial_{t}f+\eta \partial_{t}f \\[3mm]  
\partial_{x}\partial_{t}f(a,t)+\frac{1}{2}\partial_{t}\Phi(a,t)
\partial_{t}f(a,t)\\[3mm]
-\partial_{x}\partial_{t}f(b,t)+\frac{1}{2}\partial_{t}\Phi(b,t)\partial_{t}f(b,t)\end{pmatrix} \notag\\
&= \begin{pmatrix}
\partial_{t}\eta f \\[3mm] \frac{1}{2}\partial_{tt}\Phi(a,t) f(a,t) \\[3mm] \frac{1}{2}\partial_{tt}\Phi(b,t) f(b,t)
\end{pmatrix}.
\end{align*}
Then,
\begin{equation*}\label{2.25}
\begin{aligned}
-\left(\mathcal{S}^{\prime}+[\mathcal{S}, \mathcal{A}]\right) F = \begin{pmatrix}
-\partial_{t}\eta f +\partial_{xx}\left(\partial_{x} \Phi \partial_{x} f + \frac{1}{2}\partial_{xx}\Phi f \right)-\partial_{x}\Phi\partial_{xxx} f -\partial_{x}\Phi\partial_{x}\eta f -\frac{1}{2}\partial_{xx}\Phi\partial_{xx} f \\[3mm]
-\frac{1}{2}\partial_{tt}\Phi(a,t) f(a,t)+{\partial_{x}}_{|x=a}\left(\partial_{x} \Phi \partial_{x} f + \frac{1}{2}\partial_{xx}\Phi f \right)-\frac{1}{2}\partial_{x}\Phi(a,t) \partial_{x}f(a,t)\\[3mm]
-\frac{1}{2}\partial_{tt}\Phi(b,t) f(b,t)-{\partial_{x}}_{|x=b}\left(\partial_{x} \Phi \partial_{x} f + \frac{1}{2}\partial_{xx}\Phi f \right)-\frac{1}{2}\partial_{x}\Phi(b,t) \partial_{x}f(b,t)
\end{pmatrix}.
\end{aligned}
\end{equation*}
Hence,
\begin{align*}
& \left\langle-\left(\mathcal{S}^{\prime}+[\mathcal{S}, \mathcal{A}]\right) F, F\right\rangle- \partial_{x}f(b,t)\left[\mathcal{A}_{1}f(b,t)-\mathcal{A}_{3}f(b,t)\right] \nonumber\\
 &-\partial_{x}f(a,t)\left[\mathcal{A}_{2}f(a,t)-\mathcal{A}_{1}f(a,t)\right]-\partial_{x}\Phi(a,t)f(a,t)\left[\mathcal{S}_{1}f(a,t)-\mathcal{S}_{2}f(a,t)\right]\\
 &-\partial_{x}\Phi(b,t)f(b,t)\left[\mathcal{S}_{3}f(b,t)-\mathcal{S}_{1}f(b,t)\right]\\
 &=\int_{a}^b-\partial_{t}\eta |f|^2\mathrm{d}x+\int_{a}^b\partial_{xx}\left( \partial_{x} \Phi\partial_{x} f+\dfrac{1}{2}\partial_{xx} \Phi f\right)f\mathrm{d}x-\int_{a}^b\partial_{x}\Phi\partial_{xxx} f f\mathrm{d}x\\
&-\int_{a}^b\partial_{x}\Phi\partial_{x} \eta |f|^2\mathrm{d}x-\dfrac{1}{2}\int_{a}^b\partial_{xx} \Phi \partial_{xx} f f\mathrm{d}x-\frac{1}{2}\partial_{tt}\Phi(a,t) |f(a,t)|^2\\
&+{\partial_{x}}_{|x=a}\left(\partial_{x}\Phi\partial_{x}f+\frac{1}{2}\partial_{xx}\Phi f\right)f(a,t)-\frac{1}{2}\partial_{x}\Phi(a,t)\partial_{x}f(a,t)f(a,t)-\frac{1}{2}\partial_{tt}\Phi(b,t)|f(b,t)|^2\\
&-{\partial_{x}}_{|x=b}\left(\partial_{x}\Phi\partial_{x}f+\frac{1}{2}\partial_{xx}\Phi f\right)f(b,t)-\frac{1}{2}\partial_{x}\Phi(b,t)\partial_{x}f(b,t)f(b,t)\\
&+\partial_{x}\Phi(b,t)\left\lvert \partial_{x}f(b,t)\right\lvert^2+\frac{1}{2}\partial_{xx}\Phi(b,t)\partial_{x}f(b,t)f(b,t)+\frac{1}{2}\partial_{x}\Phi(b,t)\partial_{x}f(b,t)f(b,t)\\
&+\frac{1}{2}\partial_{x}\Phi(a,t)\partial_{x}f(a,t) f(a,t)-\partial_{x}\Phi(a,t)\left\lvert\partial_{x}f(a,t)\right\lvert^2-\frac{1}{2}\partial_{xx}\Phi(a,t)\partial_{x}f(a,t) f(a,t)\\
&-\partial_{x}\Phi(a,t)\partial_{xx}f(a,t) f(a,t)-\eta(a,t)\partial_{x}\Phi(a,t)|f(a,t)|^2+\partial_{x}\Phi(a,t)\partial_{x}f(a,t)f(a,t)\\
&+\frac{1}{2}\partial_{t}\Phi(a,t)\partial_{x}\Phi(a,t)|f(a,t)|^2+\partial_{x}\Phi(b,t)\partial_{x}f(b,t)f(b,t)-\frac{1}{2}\partial_{t}\Phi(b,t)\partial_{x}\Phi(b,t)\left\lvert f(b,t)\right\lvert^2\\
&+\partial_{x}\Phi(b,t)\partial_{xx}f(b,t)f(b,t)+\eta(b,t)\partial_{x}\Phi(b,t) \left\lvert f(b,t)\right\lvert^2.
\end{align*}
Using integration by parts, we obtain that
\begin{align*}
& \left\langle-\left(\mathcal{S}^{\prime}+[\mathcal{S}, \mathcal{A}]\right) F, F\right\rangle- \partial_{x}f(b,t)\left[\mathcal{A}_{1}f(b,t)-\mathcal{A}_{3}f(b,t)\right]\\
&-\partial_{x}f(a,t)\left[\mathcal{A}_{2}f(a,t)-\mathcal{A}_{1}f(a,t)\right]-\partial_{x}\Phi(a,t)f(a,t)\left[\mathcal{S}_{1}f(a,t)-\mathcal{S}_{2}f(a,t)\right]\\
&-\partial_{x}\Phi(b,t)f(b,t)\left[\mathcal{S}_{3}f(b,t)-\mathcal{S}_{1}f(b,t)\right]\\
&=\int_{a}^b -\partial_{t}\eta |f|^2\mathrm{d}x+{\partial_{x}}_{|x=b}\left(\partial_{x}\Phi \partial_{x}f+\frac{1}{2}\partial_{xx}\Phi f\right)f(b,t)\\
&-{\partial_{x}}_{|x=a}\left(\partial_{x}\Phi \partial_{x}f+\frac{1}{2}\partial_{xx}\Phi f\right)f(a,t)-\int_{a}^b \partial_{x}\left(\partial_{x}\Phi \partial_{x}f+\frac{1}{2}\partial_{xx}\Phi f\right)\partial_{x}f\mathrm{d}x\\
&-\partial_{x}\Phi(b,t)\partial_{xx}f(b,t)f(b,t)+\partial_{x}\Phi(a,t)\partial_{xx}f(a,t)f(a,t)+\int_{a}^b\partial_{xx}\Phi\partial_{xx}f f\mathrm{d}x\\
&+\frac{1}{2}\int_{a}^b \partial_{x}\Phi \partial_{x}|\partial_{x}f|^2\mathrm{d}x-\int_{a}^b \partial_{x}\Phi \partial_{x}\eta |f|^2\mathrm{d}x-\frac{1}{2}\partial_{xx}\Phi(b,t)\partial_{x}f(b,t) f(b,t)
\end{align*}
\begin{align*}
&+\frac{1}{2}\partial_{xx}\Phi(a,t)\partial_{x}f(a,t) f(a,t)+\frac{1}{4}\int_{a}^b \partial_{xxx}\Phi\partial_{x}|f|^2 \mathrm{d}x+\frac{1}{2}\int_{a}^b \partial_{xx}\Phi |\partial_{x}f|^2\mathrm{d}x\\
& -\frac{1}{2}\partial_{tt}\Phi(a,t) |f(a,t)|^2+{\partial_{x}}_{|x=a}\left(\partial_{x}\Phi\partial_{x}f+\frac{1}{2}\partial_{xx}\Phi f\right)f(a,t)\\
&-\frac{1}{2}\partial_{tt}\Phi(b,t)|f(b,t)|^2-{\partial_{x}}_{|x=b}\left(\partial_{x}\Phi\partial_{x}f+\frac{1}{2}\partial_{xx}\Phi f\right)f(b,t)\\
&+\partial_{x}\Phi(b,t)\left\lvert \partial_{x}f(b,t)\right\lvert^2+\frac{1}{2}\partial_{xx}\Phi(b,t)\partial_{x}f(b,t)f(b,t)+\partial_{x}\Phi(b,t)\partial_{x}f(b,t)f(b,t)\\
&+\partial_{x}\Phi(a,t)\partial_{x}f(a,t) f(a,t)-\partial_{x}\Phi(a,t)\left\lvert\partial_{x}f(a,t)\right\lvert^2-\frac{1}{2}\partial_{xx}\Phi(a,t)\partial_{x}f(a,t) f(a,t)\\
&-\partial_{x}\Phi(a,t)\partial_{xx}f(a,t) f(a,t)-\eta(a,t)\partial_{x}\Phi(a,t)|f(a,t)|^2+\frac{1}{2}\partial_{t}\Phi(a,t)\partial_{x}\Phi(a,t)|f(a,t)|^2\\
&-\frac{1}{2}\partial_{t}\Phi(b,t)\partial_{x}\Phi(b,t)\left\lvert f(b,t)\right\lvert^2+\partial_{x}\Phi(b,t)\partial_{xx}f(b,t)f(b,t)+\eta(b,t)\partial_{x}\Phi(b,t) \left\lvert f(b,t)\right\lvert^2\\
&=\int_{a}^b -\partial_{t}\eta |f|^2\mathrm{d}x-\int_{a}^b \partial_{xx}\Phi |\partial_{x}f|^2\mathrm{d}x-\frac{1}{2}\int_{a}^b \partial_{x}\Phi \partial_{x}|\partial_{x}f|^2\mathrm{d}x\\
&-\frac{1}{4}\int_{a}^b \partial_{xxx}\Phi \partial_{x}|f|^2\mathrm{d}x-\frac{1}{2}\int_{a}^b \partial_{xx}\Phi |\partial_{x}f|^2\mathrm{d}x+\partial_{xx}\Phi(b,t)\partial_{x}f(b,t) f(b,t)\\
&-\partial_{xx}\Phi(a,t)\partial_{x}f(a,t) f(a,t)-\frac{1}{2}\int_{a}^b \partial_{xxx}\Phi \partial_{x}|f|^2\mathrm{d}x-\int_{a}^b \partial_{xx}\Phi |\partial_{x}f|^2\mathrm{d}x\\
&+\frac{1}{2}\int_{a}^b \partial_{x}\Phi \partial_{x}|\partial_{x}f|^2\mathrm{d}x-\int_{a}^b \partial_{x}\Phi \partial_{x}\eta |f|^2\mathrm{d}x+\frac{1}{2}\int_{a}^b \partial_{xx}\Phi |\partial_{x}f|^2\mathrm{d}x\\
&-\frac{1}{2}\partial_{tt}\Phi(a,t)|f(a,t)|^2-\frac{1}{2}\partial_{tt}\Phi(b,t)|f(b,t)|^2+\partial_{x}\Phi(b,t)|\partial_{x}f(b,t)|^2-\partial_{x}\Phi(a,t)|\partial_{x}f(a,t)|^2\\
&+\partial_{x}\Phi(b,t) \partial_{x}f(b,t) f(b,t)+\partial_{x}\Phi(a,t)\partial_{x}f(a,t) f(a,t)-\eta(a,t)\partial_{x}\Phi(a,t)|f(a,t)|^2\\
&+\frac{1}{2}\partial_{t}\Phi(a,t)\partial_{x}\Phi(a,t)|f(a,t)|^2-\frac{1}{2}\partial_{t}\Phi(b,t) \partial_{x}\Phi(b,t)|f(b,t)|^2\\
&+\eta(b,t)\partial_{x}\Phi(b,t)|f(b,t)|^2\\
&=\int_{a}^b -\partial_{t}\eta |f|^2\mathrm{d}x-2\int_{a}^b \partial_{xx}\Phi |\partial_{x}f|^2\mathrm{d}x-\int_{a}^b \partial_{x}\Phi \partial_{x}\eta |f|^2\mathrm{d}x\\
&+\partial_{xx}\Phi(b,t) \partial_{x}f(b,t) f(b,t)-\partial_{xx}\Phi(a,t)\partial_{x}f(a,t) f(a,t)-\frac{1}{2}\partial_{tt}\Phi(a,t)|f(a,t)|^2\\
&-\frac{1}{2}\partial_{tt}\Phi(b,t)|f(b,t)|^2+\partial_{x}\Phi(b,t)|\partial_{x}f(b,t)|^2-\partial_{x}\Phi(a,t)|\partial_{x}f(a,t)|^2\\
&+\partial_{x}\Phi(b,t)\partial_{x}f(b,t)f(b,t)+\partial_{x}\Phi(a,t)\partial_{x}f(a,t) f(a,t)-\eta(a,t) \partial_{x}\Phi(a,t)|f(a,t)|^2\\
&+\frac{1}{2}\partial_{t}\Phi(a,t)\partial_{x}\Phi(a,t)|f(a,t)|^2-\frac{1}{2}\partial_{t}\Phi(b,t)\partial_{x}\Phi(b,t)|f(b,t)|^2+\eta(b,t)\partial_{x}\Phi(b,t)|f(b,t)|^2.
\end{align*}
Using the fact that $$\Phi\left(x,t\right)= \frac{s \varphi(x)}{\Upsilon(t)},$$
we infer that
\begin{align*}
& \left\langle-\left(\mathcal{S}^{\prime}+[\mathcal{S}, \mathcal{A}]\right) F, F\right\rangle- \partial_{x}f(b,t)\left[\mathcal{A}_{1}f(b,t)-\mathcal{A}_{3}f(b,t)\right] \nonumber -\partial_{x}f(a,t)\left[\mathcal{A}_{2}f(a,t)-\mathcal{A}_{1}f(a,t)\right]\\
&-\partial_{x}\Phi(a,t)f(a,t)\left[\mathcal{S}_{1}f(a,t)-\mathcal{S}_{2}f(a,t)\right]-\partial_{x}\Phi(b,t)f(b,t)\left[\mathcal{S}_{3}f(b,t)-\mathcal{S}_{1}f(b,t)\right]\\
&=\frac{-s(2-s)}{2\Upsilon^{3}}\int_{a}^{b}\varphi\lvert f \rvert^{2} \mathrm{d}x + \frac{s}{\Upsilon}\int_{a}^{b}\lvert \partial_{x}f \rvert^{2} \mathrm{d}x + \frac{s^{2}(2-s)}{4\Upsilon^{3}}\int_{a}^{b}\varphi\lvert f \rvert^{2}\mathrm{d}x - \frac{s}{2\Upsilon}\partial_{x}f(b,t)f(b,t)\\
&+\frac{s}{2\Upsilon}\partial_{x}f(a,t)f(a,t) - \frac{s}{\Upsilon^{3}}\varphi(a)\lvert f(a,t) \rvert^{2} - \frac{s}{\Upsilon^{3}}\varphi(b)\lvert f(b,t) \rvert^{2} + \frac{s}{\Upsilon} \partial_{x}\varphi(b)\lvert\partial_{x} f(b,t) \rvert^{2}\\
&- \frac{s}{\Upsilon} \partial_{x}\varphi(a)\lvert\partial_{x} f(a,t) \rvert^{2} + \frac{s}{\Upsilon} \partial_{x}\varphi(b)\partial_{x} f(b,t)f(b,t)+ \frac{s}{\Upsilon} \partial_{x}\varphi(a)\partial_{x} f(a,t)f(a,t)\\
&- \frac{s^2 (2-s)}{4 \Upsilon^{3}} \varphi(a)\partial_{x}\varphi(a,t)\lvert f(a,t)\rvert^{2} + \frac{s^2}{2\Upsilon^{3}}\varphi(a)\partial_{x}\varphi(a) \lvert f(a,t)\rvert^{2} - \frac{s^2}{2\Upsilon^{3}}\varphi(b)\partial_{x}\varphi(b) \lvert f(b,t)\rvert^{2}\\
&+ \frac{s^2 (2-s)}{4 \Upsilon^{3}} \varphi(b)\partial_{x}\varphi(b,t)\lvert f(b,t)\rvert^{2}\\
&=\frac{-s(2-s)^{2}}{4\Upsilon^{3}}\int_{a}^{b}\varphi\lvert f \rvert^{2} \mathrm{d}x + \frac{s}{\Upsilon}\int_{a}^{b}\lvert \partial_{x}f \rvert^{2} \mathrm{d}x + \frac{s}{2\Upsilon}(2\partial_{x}\varphi(b)-1) \partial_{x}f(b,t)f(b,t)\\
&+ \frac{s}{2\Upsilon}(2\partial_{x}\varphi(a)+1) \partial_{x}f(a,t)f(a,t) - \frac{s}{\Upsilon^{3}} \varphi(a)\lvert f(a,t) \rvert^{2} - \frac{s}{\Upsilon^{3}} \varphi(b)\lvert f(b,t) \rvert^{2}\\
&+\frac{s}{\Upsilon} \partial_{x}\varphi(b)\lvert\partial_{x} f(b,t) \rvert^{2}
- \frac{s}{\Upsilon}\partial_{x}\varphi(a)\lvert\partial_{x} f(a,t) \rvert^{2} + \frac{s^{3}}{4\Upsilon^{3}}\varphi(a)\partial_{x}\varphi(a)\lvert f(a,t) \rvert^{2}\\
&- \frac{s^{3}}{4\Upsilon^{3}}\varphi(b)\partial_{x}\varphi(b)\lvert f(b,t) \rvert^{2}.
\end{align*}
Therefore, we obtain the desired equality \eqref{9.1}.

\noindent\textbf{Step 4.}
For any $\hbar\in(0,1]$ and $ 0<s\leq \min\left(\dfrac{2}{\sqrt{x_{0}-a}},\dfrac{2}{\sqrt{b-x_{0}}} ,1\right) ,$ we prove that
\begin{align}\label{14.11}
& \left\langle-\left(\mathcal{S}^{\prime}+[\mathcal{S}, \mathcal{A}]\right) F, F\right\rangle- \partial_{x}f(b,t)\left[\mathcal{A}_{1}f(b,t)-\mathcal{A}_{3}f(b,t)\right] \nonumber\\\nonumber
 &-\partial_{x}f(a,t)\left[\mathcal{A}_{2}f(a,t)-\mathcal{A}_{1}f(a,t)\right]-\partial_{x}\Phi(a,t)f(a,t)\left[\mathcal{S}_{1}f(a,t)-\mathcal{S}_{2}f(a,t)\right]\\\nonumber
 &-\partial_{x}\Phi(b,t)f(b,t)\left[\mathcal{S}_{3}f(b,t)-\mathcal{S}_{1}f(b,t)\right]\\
 &\leq (1+C_{0}) \left< -SF,F \right> + \frac{C}{\hbar^{2}} \|F\|^{2},
 \end{align}
where $$C=\max \left(\frac{(\partial_{x}\varphi(b)-1)^2}{4(b-x_{0})},\frac{(\partial_{x}\varphi(a)+1)^2}{4(x_{0}-a)}\right),\;
 C_{0} = 1 - \min\left(s,\frac{s^{2}}{4}(x_{0}-a),\frac{s^{2}}{4}(b-x_{0}) \right) \in (0,1).$$
Indeed, we have
\begin{align*}
\left \langle\mathcal{S}F,F\right \rangle &= -\int_{a}^{b} \lvert\partial_{x} f(x,t) \rvert^{2} \mathrm{d}x + \frac{s(2-s)}{4\Upsilon^{2}} \int_{a}^{b}\varphi (x) \lvert f(x,t) \rvert^{2} \mathrm{d}x\\
&+ \frac{s}{2\Upsilon^{2}}\varphi(a) \lvert f(a,t) \rvert^{2} + \frac{s}{2\Upsilon^{2}} \varphi(b) \lvert f(b,t) \rvert^{2}.
\end{align*}
Then
\begin{align}\label{11.1}
&\dfrac{1}{\Upsilon}\left\langle-\mathcal{S}F,F\right\rangle=\frac{1}{\Upsilon}\int_{a}^b |\partial_{x}f|^2\mathrm{d}x+\frac{s(2-s)}{4\Upsilon^3}\int_{a}^b\left( -\varphi\right) |f|^2\mathrm{d}x\nonumber\\
&+\frac{s}{2\Upsilon^3}\left(-\varphi(a)\right)|f(a,t)|^2+\frac{s}{2\Upsilon^3}\left(-\varphi(b)\right)|f(b,t)|^2.
\end{align}
Next, we estimate each term appearing in equality \eqref{9.1}. For $ 0<s\leq \min\left(\dfrac{2}{\sqrt{x_{0}-a}},\dfrac{2}{\sqrt{b-x_{0}}} ,1\right) ,$ we have
\begin{equation}\label{12.1}
    \frac{-s(2-s)^{2}}{4\Upsilon^{3}}\int_{a}^{b}\varphi\lvert f \rvert^{2} \mathrm{d}x =(2-s)\left[\frac{s(2-s)}{4\Upsilon^{3}}\int_{a}^{b}\left(-\varphi\right)\lvert f \rvert^{2} \mathrm{d}x \right].
\end{equation}
Since $s\in (0,1)$, we have
\begin{equation}\label{13.1}
    \frac{s}{\Upsilon}\int_{a}^{b}\lvert \partial_{x}f \rvert^{2} \mathrm{d}x \leq \frac{1}{\Upsilon}\int_{a}^{b}\lvert \partial_{x}f \rvert^{2} \mathrm{d}x,
\end{equation}
and
\begin{align}\label{14.1}
  &\frac{-s}{\Upsilon^{3}}\varphi(a)\lvert f(a,t) \rvert^{2}+  \frac{s^{3}}{4\Upsilon^{3}}\varphi(a)\partial_{x}\varphi(a)\lvert f(a,t) \rvert^{2}\nonumber\\
  &= \left(2 - \frac{s^{2}}{2} \partial_{x}\varphi(a) \right) \left[ \frac{s}{2\Upsilon^{3}} \left(-\varphi(a)\right) \lvert f(a,t) \rvert^{2}\right]\nonumber\\
  &= \left(2 - \frac{s^{2}}{4} (x_{0}-a) \right) \left[ \frac{s}{2\Upsilon^{3}} \left(-\varphi(a)\right) \lvert f(a,t) \rvert^{2}\right].
\end{align}
In the same manner we obtain that
\begin{align}\label{15.1}
  &\frac{-s}{\Upsilon^{3}}\varphi(b)\lvert f(b,t) \rvert^{2}-  \frac{s^{3}}{4\Upsilon^{3}}\varphi(b)\partial_{x}\varphi(b)\lvert f(b,t) \rvert^{2}\nonumber\\
  &= \left(2 - \frac{s^{2}}{4} (b-x_{0}) \right) \left[ \frac{s}{2\Upsilon^{3}} \left(-\varphi(b)\right) \lvert f(b,t) \rvert^{2}\right].
\end{align}
Since $x_{0} \in (a,b)$, we have
\begin{align}\label{16.1}
\frac{s}{\Upsilon}\partial_{x}\varphi(b)\lvert\partial_{x} f(b,t) \rvert^{2} - \frac{s}{\Upsilon}\partial_{x}\varphi(a)\lvert\partial_{x} f(a,t) \rvert^{2}&= \frac{-s}{2\Upsilon}(b-x_{0})\lvert\partial_{x} f(b,t) \rvert^{2}+\frac{s}{2\Upsilon}(a-x_{0})\lvert\partial_{x} f(a,t) \rvert^{2}\le 0.
\end{align}
Furthermore, we have 
\begin{align*}
 \frac{s}{2\Upsilon}(2\partial_{x}\varphi(b)-1) \partial_{x}f(b,t)f(b,t) =   \frac{s}{2\Upsilon} \left[\left(\sqrt{b-x_{0}}\partial_{x}f(b,t) \right)\left(\frac{(2\partial_{x}\varphi(b)-1)}{\sqrt{b-x_{0}}} f(b,t)\right) \right].
\end{align*}
Using Young's inequality, we obtain
\begin{align}\label{17.1}
 \frac{s}{2\Upsilon}(2\partial_{x}\varphi(b)-1) \partial_{x}f(b,t)f(b,t) \leq  \frac{s}{4\Upsilon} (b-x_0) \lvert \partial_{x}f(b,t) \rvert^{2} + \frac{s(2\partial_{x}\varphi(b)-1)^{2}}{4\Upsilon( b-x_{0})} \lvert f(b,t) \rvert^{2}.
\end{align}
In the same manner, we obtain
\begin{align}\label{18.1}
 \frac{s}{2\Upsilon}(2\partial_{x}\varphi(a)+1) \partial_{x}f(a,t)f(a,t) \leq  \frac{s}{4\Upsilon} (x_0 - a) \lvert \partial_{x}f(a,t) \rvert^{2} + \frac{s(2\partial_{x}\varphi(a)-1)^{2}}{4\Upsilon( x_{0}-a)} \lvert f(a,t) \rvert^{2}.
\end{align}
Combining \eqref{9.1} and \eqref{11.1}-\eqref{18.1}, 
we infer that
\begin{align*}
& \left\langle-\left(\mathcal{S}^{\prime}+[\mathcal{S}, \mathcal{A}]\right) F, F\right\rangle- \partial_{x}f(b,t)\left[\mathcal{A}_{1}f(b,t)-\mathcal{A}_{3}f(b,t)\right] \nonumber\\
 &-\partial_{x}f(a,t)\left[\mathcal{A}_{2}f(a,t)-\mathcal{A}_{1}f(a,t)\right]-\partial_{x}\Phi(a,t)f(a,t)\left[\mathcal{S}_{1}f(a,t)-\mathcal{S}_{2}f(a,t)\right]\\
 &-\partial_{x}\Phi(b,t)f(b,t)\left[\mathcal{S}_{3}f(b,t)-\mathcal{S}_{1}f(b,t)\right]\\
 &\leq (2-s)\left[\frac{s(2-s)}{4\Upsilon^{3}}\int_{a}^{b}\left(-\varphi\right)\lvert f \rvert^{2} \mathrm{d}x \right] + \frac{1}{\Upsilon}\int_{a}^{b}\lvert \partial_{x}f \rvert^{2} \mathrm{d}x \\
 &+\left(2 - \frac{s^{2}}{4} (x_{0}-a) \right) \left[ \frac{s}{2\Upsilon^{3}} \left(-\varphi(a)\right) \lvert f(a,t) \rvert^{2}\right] + \left(2 - \frac{s^{2}}{4} (b-x_{0}) \right) \left[ \frac{s}{2\Upsilon^{3}} \left(-\varphi(b)\right) \lvert f(b,t) \rvert^{2}\right]\\
 &-\frac{s}{2\Upsilon} (b-x_0) \lvert \partial_{x}f(b,t) \rvert^{2} - \frac{s}{2\Upsilon} (x_0 -a) \lvert \partial_{x}f(a,t) \rvert^{2}\\
 &+\frac{s(2\partial_{x}\varphi(b)-1)^{2}}{4\Upsilon( b-x_{0})} \lvert f(b,t) \rvert^{2} + \frac{s(2\partial_{x}\varphi(a)+1)^{2}}{4\Upsilon( x_{0}-a)} \lvert f(a,t) \rvert^{2}.
 \end{align*}
Since $\hbar\in (0,1)$, we have $\dfrac{1}{\Upsilon} \leq \dfrac{1}{\hbar}\leq \dfrac{1}{\hbar^{2}}$. Then, for $C=\max \left(\dfrac{(\partial_{x}\varphi(b)-1)^2}{4(b-x_{0})},\dfrac{(\partial_{x}\varphi(a)+1)^2}{4(x_{0}-a)}\right) \text{ and}$
$$C_{0} = 1 - \min\left(s,\frac{s^{2}}{4}(x_{0}-a),\frac{s^{2}}{4}(b-x_{0}) \right) \in (0,1),$$
we obtain
\begin{align*}
& \left\langle-\left(\mathcal{S}^{\prime}+[\mathcal{S}, \mathcal{A}]\right) F, F\right\rangle- \partial_{x}f(b,t)\left[\mathcal{A}_{1}f(b,t)-\mathcal{A}_{3}f(b,t)\right] \nonumber\\
&-\partial_{x}f(a,t)\left[\mathcal{A}_{2}f(a,t)-\mathcal{A}_{1}f(a,t)\right]-\partial_{x}\Phi(a,t)f(a,t)\left[\mathcal{S}_{1}f(a,t)-\mathcal{S}_{2}f(a,t)\right]\\
&-\partial_{x}\Phi(b,t)f(b,t)\left[\mathcal{S}_{3}f(b,t)-\mathcal{S}_{1}f(b,t)\right]\\
&\leq (1+C_{0}) \Bigg[\frac{s(2-s)}{4\Upsilon^{3}}\int_{a}^{b}\left(-\varphi\right)\lvert f \rvert^{2} \mathrm{d}x + \frac{1}{\Upsilon}\int_{a}^{b}\lvert \partial_{x}f \rvert^{2} \mathrm{d}x \\
&+  \frac{s}{2\Upsilon^{3}} \left(-\varphi(a)\right) \lvert f(a,t) \rvert^{2} +  \frac{s}{2\Upsilon^{3}} \left(-\varphi(b)\right) \lvert f(b,t) \rvert^{2}\Bigg] +\frac{C}{\hbar^{2}}\left( \lvert f(b,t) \rvert^{2} +  \lvert f(a,t) \rvert^{2}\right)\\
&\leq (1+C_{0}) \left< -S,F \right> + \frac{C}{\hbar^{2}} \|F\|^{2}.
\end{align*}
 
\noindent\textbf{Step 5.} The following differential system holds
\begin{empheq}[left = \empheqlbrace]{alignat=2}
\begin{aligned}
&\frac{1}{2}\frac{\d}{\d t}\left\Vert F\left(  \cdot,t\right)
\right\Vert ^{2}+\mathcal{N}\left(t\right)  \left\Vert F\left(
\cdot,t\right)  \right\Vert ^{2}=0, & \\
&\frac{\d}{\d t}\mathcal{N}\left(  t\right)  \leq\frac
{1+C_{0}}{\Upsilon\left(t\right)  }\mathcal{N}\left(  t\right)+\frac{C}{\hbar^2}.
\end{aligned}
\end{empheq}
Using \cite{RBKDP}, Proposition 3, we infer, for any $0<t_{1}<t_{2}<t_{3}\leq T$, that
\[
\left(\left\Vert F\left(\cdot,t_{2}\right)  \right\Vert ^{2}\right)
^{1+M}\leq\left(  \left\Vert F\left(  \cdot,t_{1}\right)  \right\Vert
^{2}\right)  ^{M}\left\Vert F\left(  \cdot,t_{3}\right)  \right\Vert ^{2}\mathrm{e}^{D},%
\]
where%
\[
M=\dfrac{\displaystyle\int_{t_{2}}^{t_{3}}\dfrac{\mathrm{d}t}{(T-t+\hbar)^{1+C_{0}}} }{\displaystyle\int_{t_{1}}^{t_{2}}\dfrac{\mathrm{d}t}{(T-t+\hbar)^{1+C_{0}}}} \qquad \text{ and } \qquad D=2(1+M)(t_{3}-t_{1})^2 \frac{C}{\hbar^2}. %
\]
Thus, we obtain
\[
\left( \left\| U\left(  \cdot,t_{2}\right) \mathrm{e}^{\frac{\Phi\left(  \cdot,t_{2}\right)}{2}} \right\| ^{2}  \right)  ^{1+M}\leq\left(
\left\| U\left(  \cdot,t_{1}\right) \mathrm{e}^{\frac{\Phi\left(  \cdot,t_{1}\right)}{2}} \right\| ^{2}\right)^{M}\left\| U\left(  \cdot,t_{3}\right) \mathrm{e}^{\frac{\Phi\left(  \cdot,t_{3}\right)}{2}} \right\| ^{2} \mathrm{e}^D\text{ .}%
\]

\noindent\textbf{Step 6.} We take off the weight function $\Phi$ from the integrals
\begin{equation}\label{2.36}
\begin{array}
[c]{ll}%
\left(\|U\left(  \cdot,t_{2}\right) \|^{2}  \right)  ^{1+M}&\leq\exp\left[  -\left(  1+M\right)
\min\limits_{x\in [a,b]}\Phi\left(  x,t_{2}\right)
+M\max\limits_{x\in [a,b]}\Phi\left(x,t_{1}\right)
\right]
\\
& \quad \times\left(
\left\| U\left(  \cdot,t_{1}\right)  \right\| ^{2}\right)^{M}\left\|U\left(  \cdot,t_{3}\right) \mathrm{e}^{\frac{\Phi\left(  \cdot,t_{3}\right)}{2}} \right\|^{2} \mathrm{e}^D\text{ .}%
\end{array}
\end{equation}
Let $\omega \Subset (a,b)$ be a nonempty open subset. Then
\[
\begin{array}[c]{ll}
\left\lVert U\left(\cdot,t_{3}\right) \mathrm{e}^{\frac{\Phi\left(  \cdot,t_{3}\right)}{2}} \right\rVert ^{2}&= \displaystyle\int_{a}^b\left\vert u\left(  x,t_{3}\right)  \right\vert^{2}\mathrm{e}^{\Phi\left(  x,t_{3}\right)  } \d x+\left\vert u\left(  a,t_{3}\right)  \right\vert
^{2}\mathrm{e}^{\Phi\left(  a,t_{3}\right) }+\left\vert u\left(  b,t_{3}\right)  \right\vert
^{2}\mathrm{e}^{\Phi\left(  b,t_{3}\right) }\\
&= \displaystyle\int_{\omega}\left\vert u\left(  x,t_{3}\right)  \right\vert
^{2}\mathrm{e}^{\Phi\left(  x,t_{3}\right)  }\mathrm{d} x+\displaystyle\int_{\left.  (a,b)\right\backslash \omega}\left\vert u\left(  x,t_{3}\right)  \right\vert^{2}\mathrm{e}^{\Phi\left(  x,t_{3}\right)  }\mathrm{d}x\\
&\quad+\left\vert u\left(  a,t_{3}\right)  \right\vert
^{2}\mathrm{e}^{\Phi\left(  a,t_{3}\right) }+\left\vert u\left(  b,t_{3}\right)  \right\vert
^{2}\mathrm{e}^{\Phi\left(  b,t_{3}\right) }\\
& \leq\exp\left[ \max\limits_{x\in\overline{\omega}}
\Phi\left(  x,t_{3}\right)  \right]  \displaystyle\int_{\omega}\left\vert
u\left(  x,t_{3}\right)  \right\vert ^{2}\mathrm{d}x\\
&\quad+\exp\left[\max\limits_{x\in\overline{\left.  (a,b)
\right\backslash \omega}}\Phi\left(  x,t_{3}\right)  \right]
\displaystyle\int_{a}^b\left\vert u\left(  x,t_{3}\right)  \right\vert
^{2}\mathrm{d}x\\
&\quad+\left\vert u\left(  a,t_{3}\right)  \right\vert
^{2}\mathrm{e}^{\Phi\left(  a,t_{3}\right) }+\left\vert u\left(  b,t_{3}\right)  \right\vert
^{2}\mathrm{e}^{\Phi\left(  b,t_{3}\right) }.
\end{array}
\]
Since $\omega \Subset (a,b)$, then 
\[\exp\left[  \max\limits_{x\in\{a,b\}} \Phi\left(  x,t_{3}\right)  \right]\leq\exp\left[\max\limits_{x\in\overline{\left.  (a,b)
\right\backslash \omega}}\Phi\left(  x,t_{3}\right)  \right]. \]
Therefore,
\begin{equation}\label{2.37}
\left\lVert U\left(  \cdot,t_{3}\right) \mathrm{e}^{\frac{\Phi\left(  \cdot,t_{3}\right)}{2}} \right\rVert ^{2}\leq \exp\left[\max\limits_{x\in\overline{\omega}}\Phi\left(  x,t_{3}\right)  \right]  \displaystyle\int_{\omega}\left\vert
u\left(  x,t_{3}\right)  \right\vert ^{2}\mathrm{d}x+\exp\left[\max\limits_{x\in\overline{\left.  (a,b)\right\backslash \omega}}\Phi\left(  x,t_{3}\right)  \right]\| U(\cdot,t_{3})\|^2.
\end{equation}
Using \eqref{2.36}-\eqref{2.37}, we obtain
\[
\begin{array}
[c]{ll}%
\left( \| U\left(  \cdot,t_{2}\right)  \| ^{2}  \right)  ^{1+M}&\leq \mathrm{e}^D\exp\left[  -\left(  1+M\right)
\min\limits_{x\in[a,b]}\Phi\left(  x,t_{2}\right)
+M\max\limits_{x\in[a,b]}\Phi\left(  x,t_{1}\right)
+\max\limits_{x\in\overline{\omega}}\Phi\left(  x,t_{1}\right)\right]
\\
& \quad \times\left(
\| U\left(  \cdot,t_{1}\right)  \| ^{2}\right)  ^{M}\displaystyle\int_{\omega}\left\vert
u\left(  x,t_{3}\right)  \right\vert ^{2}\mathrm{d}x \\
&+ \mathrm{e}^D\exp\left[  -\left(  1+M\right)
\min\limits_{x\in[a,b]}\Phi\left(  x,t_{2}\right)
+M\max\limits_{x\in[a,b]}\Phi\left(  x,t_{1}\right)
+\max\limits_{x\in\overline{\left.  [a,b]\right\backslash \omega}}\Phi\left(  x,t_{1}\right)\right]\\
&\quad  \times\left(
\| U\left(  \cdot,t_{1}\right)  \| ^{2}\right)  ^{M}\|U(\cdot,t_{3})\|^2 .
\end{array}
\]
Using the fact that $\left\Vert U\left(  \cdot,T\right)  \right\Vert
\leq\left\Vert U\left(  \cdot,t\right)  \right\Vert \leq\left\Vert U\left(
\cdot,0\right)  \right\Vert ,$ $ 0<t<T$,\, the above inequality becomes%
\[%
\begin{array}
[c]{ll}%
\left(  \left\Vert U\left(  \cdot,T\right)  \right\Vert ^{2}\right)  ^{1+M} &
\leq \mathrm{e}^D\exp\left[  -\left(  1+M\right) \min\limits_{x\in[a,b]%
} \Phi\left(  x,t_{2}\right)  +M\max\limits_{x\in[a,b]%
}\Phi\left(  x,t_{1}\right)  +\max\limits_{x\in\overline{\omega}%
}\Phi\left(  x,t_{3}\right)  \right] \\
& \quad\times\left(  \left\Vert U\left(  \cdot,0\right)  \right\Vert
^{2}\right)  ^{M}\displaystyle\int_{\omega}\left\vert u\left(  x,t_{3}\right)
\right\vert ^{2}\mathrm{d}x\\
&  +\mathrm{e}^D\exp\left[  -\left(  1+M\right)  \min\limits_{x\in[a,b]}\Phi\left(  x,t_{2}\right)  +M\max\limits_{x\in[a,b]%
}\Phi\left(  x,t_{1}\right)  +\max\limits_{x\in\overline{\left.
(a,b)\right\backslash \omega}}\Phi\left(  x,t_{3}\right)  \right]
\\
& \quad\times\left(  \left\Vert U\left(  \cdot,0\right)  \right\Vert
^{2}\right)  ^{1+M}.%
\end{array}
\]
Since  $\Phi\left(  x,t\right)  =\displaystyle\frac{s\varphi\left(x\right)  }{T-t+\hbar}$, then
\[%
\begin{array}
[c]{ll}%
\left\Vert U\left(  \cdot,T\right)  \right\Vert ^{1+M} & \leq \mathrm{e}^D\exp%
\frac{s}{2}\left[  -\frac{1+M}{T-t_{2}+\hbar}\min\limits_{x\in[a,b]%
}\varphi\left(  x\right)  +\frac{M}{T-t_{1}+\hbar}\max\limits_
{x\in[a,b]}\varphi\left(  x\right)  +\frac{1}%
{T-t_{3}+\hbar}\max\limits_{x\in\overline{\omega}}\varphi\left(
x\right)  \right] \\
& \quad\times\left\Vert U\left(  \cdot,0\right)  \right\Vert ^{M}\left\Vert
u\left(  \cdot,t_{3}\right)  \right\Vert _{L^{2}\left(  \omega\right)  }\\
&+\mathrm{e}^D\exp\frac{s}{2}\left[  -\frac{1+M}{T-t_{2}+\hbar}\min\limits_
{x\in[a,b]}\varphi\left(  x\right)  +\frac{M}%
{T-t_{1}+\hbar}\max\limits_{x\in[a,b]}\varphi\left(
x\right)  +\frac{1}{T-t_{3}+\hbar}\max\limits_{x\in\overline{\left.
(a,b)\right\backslash \omega}} \varphi\left(  x\right)  \right] \\
& \quad\times\left\Vert U\left(  \cdot,0\right)  \right\Vert ^{1+M}\text{ .}%
\end{array}
\]

\noindent\textbf{Step 7.} We choose $t_{3}=T$, $t_{2}=T-\ell \hbar$, $t_{1}=T-2\ell \hbar$, and $\ell>1$ such that $0<2\ell
\hbar<T.$ Then%
\[%
\begin{array}
[c]{ll}%
\left\Vert u\left(  \cdot,T\right)  \right\Vert ^{1+M_{\ell}} & \leq
 \mathrm{e}^{D_{\ell}}\exp\frac{s}{2h}\left[  -\frac{1+M_{\ell}}{1+\ell}\min\limits_
{x\in[a,b]}\varphi\left(  x\right)  +\frac{M_{\ell}%
}{1+2\ell}\max\limits_{x\in[a,b]}\varphi\left(  x\right)
+\max\limits_{x\in\overline{\omega}}\varphi\left(  x\right)  \right]
\\
& \quad\times\left\Vert U\left(  \cdot,0\right)  \right\Vert ^{M_{\ell}%
}\left\Vert u\left(  \cdot,T\right)  \right\Vert _{L^{2}\left(  \omega\right)
}\\
& +\mathrm{e}^{D_{\ell}}\exp\frac{s}{2h}\left[  -\frac{1+M_{\ell}}{1+\ell}%
\min\limits_{x\in[a,b]}\varphi\left(  x\right)
+\frac{M_{\ell}}{1+2\ell}\max\limits_{x\in[a,b]}%
\varphi\left(  x\right)  +\max\limits_{x\in\overline{\left. (a,b)
\right\backslash \omega}}\varphi\left(  x\right)  \right] \\
& \quad\times\left\Vert U\left(  \cdot,0\right)  \right\Vert ^{1+M_{\ell}%
}\text{ ,}%
\end{array}
\]
where $M_\ell = \dfrac{(\ell+1)^{C_{0}}-1}{1-\left(\dfrac{\ell+1}{2\ell+1}\right)^{C_{0}}}$ and $D_{\ell}=2C\ell^2(1+M_{\ell})$. \\
Since $\varphi \leq 0$ and $\varphi(x_{0})=0$, then $\max\limits_{x\in[a,b]}%
\varphi\left(  x\right)=0$.  
Therefore, 
\[
 -\frac{1+M_{\ell}}{1+\ell}%
 \min\limits_{x\in[a,b]}\varphi\left(  x\right)
+\frac{M_{\ell}}{1+2\ell}\max\limits_{x\in[a,b]}
\varphi\left(  x\right)  +\max\limits_{x\in\overline{\left.  (a,b)
\right\backslash \omega}}\varphi\left(  x\right)= -\frac{1+M_{\ell}}{1+\ell}%
\min\limits_{x\in[a,b]}\varphi\left(  x\right)
+\max\limits_{x\in\overline{\left.  (a,b)
\right\backslash \omega}}\varphi\left(  x\right).
\]
Now, we use the fact that $C_{0} \in \left(0,1\right)$ and choose $\ell>1$ sufficiently large so that
\[
 -\frac{1+M_{\ell}}{1+\ell}%
\min\limits_{x\in[a,b]}\varphi\left(  x\right)
+\frac{M_{\ell}}{1+2\ell}\max\limits_{x\in[a,b]}%
\varphi\left(  x\right)  +\max\limits_{x\in\overline{\left.  (a,b)
\right\backslash \omega}}\varphi\left(  x\right) < 0.
\]
Consequently, there exist $C_{1}>0$ and $C_{2}>0$ such that for any $\hbar>0$
with $0<2\ell \hbar<T$,
\[
\left\Vert U\left(  \cdot,T\right)  \right\Vert ^{1+M_{\ell}}\leq
\mathrm{e}^{D_{\ell}} \mathrm{e}^{C_{1}\frac{1}{\hbar}}\left\Vert U\left(  \cdot,0\right)  \right\Vert
^{M_{\ell}}\left\Vert u\left(  \cdot,T\right)  \right\Vert _{L^{2}\left(
\omega\right)  }+ \mathrm{e}^{D_{\ell}} \mathrm{e}^{-C_{2}\frac{1}{\hbar}}\left\Vert U\left(  \cdot,0\right)
\right\Vert ^{1+M_{\ell}}.%
\]
Therefore, we obtain 
\begin{equation}\label{42}
\left\Vert U\left(\cdot,T\right)  \right\Vert ^{1+M_{\ell}}\leq
\mathrm{e}^{D_{\ell}} \mathrm{e}^{C_{1}\frac{1}{\hbar}}\left\Vert U\left(  \cdot,0\right)  \right\Vert
^{M_{\ell}}\left\Vert u\left(  \cdot,T\right)  \right\Vert _{L^{2}\left(
\omega\right)  } + \mathrm{e}^{D_{\ell}} \mathrm{e}^{-C_{2}\frac{1}{\hbar}}\left\Vert U\left(  \cdot,0\right)
\right\Vert ^{1+M_{\ell}}.%
\end{equation}
For any $2\ell \hbar\geq T$, we have $1\leq
\mathrm{e}^{D_{\ell}}\mathrm{e}^{C_{2}\frac{2\ell}{T}}\mathrm{e}^{-C_{2}\frac{1}{\hbar}}$. Using the fact that $\left\Vert U\left(  \cdot,T\right)  \right\Vert \leq\left\Vert
U\left(  \cdot,0\right)  \right\Vert ,$ we deduce, for any $2\ell \hbar\geq T$, that%
\begin{equation}\label{23.}
\left\Vert U\left(  \cdot,T\right)  \right\Vert ^{1+M_{\ell}}\leq \mathrm{e}^{D_{\ell}}
\mathrm{e}^{C_{1}\frac{1}{\hbar}}\left\Vert U\left(  \cdot,0\right)  \right\Vert
^{M_{\ell}}\left\Vert u\left(  \cdot,T\right)  \right\Vert _{L^{2}\left(
\omega\right)  }+\mathrm{e}^{D_{\ell}}\mathrm{e}^{C_{2}\frac{2\ell}{T}}\mathrm{e}^{-C_{2}\frac{1}{\hbar}}\left\Vert
U\left(\cdot,0\right)  \right\Vert ^{1+M_{\ell}}\text{ .}%
\end{equation}
Using \eqref{42}-\eqref{23.}, for any $\hbar >0$, we obtain that
\begin{equation*}
\left\Vert U\left(  \cdot,T\right)  \right\Vert ^{1+M_{\ell}}\leq \mathrm{e}^{D_{\ell}}
\mathrm{e}^{C_{1}\frac{1}{\hbar}}\left\Vert U\left(  \cdot,0\right)  \right\Vert
^{M_{\ell}}\left\Vert u\left(  \cdot,T\right)  \right\Vert _{L^{2}\left(
\omega\right)  }+\mathrm{e}^{D_{\ell}}\mathrm{e}^{C_{2}\frac{2\ell}{T}}\mathrm{e}^{-C_{2}\frac{1}{\hbar}}\left\Vert
U\left(\cdot,0\right)  \right\Vert ^{1+M_{\ell}}\text{ .}%
\end{equation*}
Finally, we choose $\hbar>0$ such that
\[
\mathrm{e}^{D_{\ell}}\mathrm{e}^{C_{2}\frac{2\ell}{T}}\mathrm{e}^{-C_{2}\frac{1}{\hbar}}\left\Vert U\left(
\cdot,0\right)  \right\Vert ^{1+M_{\ell}}=\frac{1}{2}\left\Vert U\left(
\cdot,T\right)  \right\Vert ^{1+M_{\ell}}\text{ ,}%
\]
that is,
\[
\mathrm{e}^{C_{2}\frac{1}{\hbar}}=2\mathrm{e}^{D_{\ell}}\mathrm{e}^{C_{2}\frac{2\ell}{T}}\left(  \frac{\left\Vert
U\left(  \cdot,0\right)  \right\Vert }{\left\Vert U\left(  \cdot,T\right)
\right\Vert }\right)^{1+M_{\ell}},%
\]
in order that%
\[
\left\Vert U\left(  \cdot,T\right)  \right\Vert^{1+M_{\ell}}\leq 2\mathrm{e}^{D_{\ell}}\left(
2\mathrm{e}^{D_{\ell}}\mathrm{e}^{C_{2}\frac{2\ell}{T}}\left(  \frac{\left\Vert U\left(\cdot,0\right)
\right\Vert }{\left\Vert U\left(\cdot,T\right)  \right\Vert }\right)
^{1+M_{\ell}}\right)  ^{\frac{C_{1}}{C_{2}}}\left\Vert U\left(  \cdot
,0\right)  \right\Vert ^{M_{\ell}}\left\Vert u\left(  \cdot,T\right)
\right\Vert _{L^{2}\left(  \omega\right)  }.
\]
Hence
\[
\left\Vert U\left(  \cdot,T\right)  \right\Vert^{1+M_{\ell}+\left(1+M_{\ell}\right) \frac{C_{1}}{C_{2}}} \leq2^{1+\frac{C_{1}}{C_{2}}}%
\mathrm{e}^{D_{\ell}\left(1+\frac{C_{1}}{C_{2}}\right)}\mathrm{e}^{C_{1}\frac{2\ell}{T}}\left(\left\Vert U\left(\cdot,0\right)
\right\Vert \right)
^{M_{\ell}+\left(  1+M_{\ell}\right)  \frac{C_{1}}{C_{2}}}\left\Vert u\left(
\cdot,T\right)  \right\Vert _{L^{2}\left(  \omega\right)  }.
\]
Setting $\sigma =M_{\ell}+\left(1+M_{\ell}\right) \frac{C_{1}}{C_{2}},$ $\mu =  2^{1+\frac{C_{1}}{C_{2}}}%
\mathrm{e}^{D_{\ell}\left(1+\frac{C_{1}}{C_{2}}\right)}$ and $k = 2\ell C_{1},$ we obtain
\[
\left\Vert U\left(  \cdot,T\right) \right\Vert \leq \left( \mu \mathrm{e}^{\frac{k}{T}}\right)^{\frac{1}{1+\sigma}}  \left\Vert U\left(\cdot,0\right) \right\Vert^{\frac{\sigma}{1+\sigma}}\left\Vert u\left(
\cdot,T\right)  \right\Vert _{L^{2}\left(  \omega\right)}^{\frac{1}{1+\sigma}  }\text{ .}%
\]
This completes the proof of Lemma \ref{lem1.1}.
\end{proof}

The following lemma is needed to establish the impulse approximate controllability of system \eqref{1.1}.
\begin{lemma}\label{lemma3.1}
Let $\vartheta=\left(\upsilon,\upsilon(a,\cdot), \upsilon(b,\cdot)\right)$ be the solution of \eqref{1.3}. Then there exist positive constants $\mathcal{M}_{1}$, $\mathcal{M}_{2}$ and $\delta=\delta\left(a,b,\omega\right)$ such that, for all $\varepsilon>0$, the following inequality holds
\begin{equation}\label{2.38.}
\left\Vert \vartheta \left(\cdot,T\right)\right\Vert^2\leq \left(\frac{\mathcal{M}_{1}\mathrm{e}^{\frac{\mathcal{M}_{2}}{T}}}{\varepsilon^{\delta}}\right)^{2}\left\Vert \upsilon\left(\cdot,T\right)\right \Vert^2_{L^2(\omega)}+\varepsilon^2\left \Vert \vartheta^0 \right\Vert ^2.
\end{equation}
\end{lemma}
\begin{proof}
Using Lemma \ref{lem1.1}, there exist $\mathcal{K}_{1}>0$, $\mathcal{K}_{2}>0$ and $\beta\in \left(0,1\right)$ such that
\begin{equation*}
\left\Vert \vartheta\left(\cdot,T\right)\right\Vert^2\leq \left(\mathcal{K}_{1}\mathrm{e}^\frac{\mathcal{K}_{2}}{T}\right)^2 \left\Vert \upsilon\left(\cdot,T\right)\right\Vert_{L^2(\omega)}^{2\beta}\left \Vert \vartheta\left(\cdot,0\right)\right\Vert^{2(1-\beta)}.
\end{equation*}
Let $\varepsilon>0$. We have 
\begin{align*}
\left(\mathcal{K}_{1}\mathrm{e}^\frac{\mathcal{K}_{2}}{T}\right)^2 \left\Vert \upsilon\left(\cdot,T\right)\right\Vert_{L^2(\omega)}^{2\beta}\left \Vert \vartheta\left(\cdot,0\right)\right\Vert^{2(1-\beta)}&=\left(\left(\mathcal{K}_{1}\mathrm{e}^\frac{\mathcal{K}_{2}}{T}\right)^\frac{1}{\beta}\left\Vert\upsilon(\cdot,T)\right\Vert_{L^2(\omega)}\frac{1}{\varepsilon^\frac{1-\beta}{\beta}}\left(1-\beta\right)^\frac{1-\beta}{2\beta}\right)^{2\beta}\\
&\quad \times \left( \varepsilon \left(\frac{1}{1-\beta}\right)^\frac{1}{2}\left\Vert \vartheta\left(\cdot,0\right)\right\Vert\right)^{2(1-\beta)}.
\end{align*}
Applying Young's inequality, we obtain
\[
\begin{array}[c]{ll}
\left(\mathcal{K}_{1}\mathrm{e}^\frac{\mathcal{K}_{2}}{T}\right)^2 \left\Vert \upsilon\left(\cdot,T\right)\right\Vert_{L^2(\omega)}^{2\beta}\left \Vert \vartheta\left(\cdot,0\right)\right\Vert^{2(1-\beta)}&\leq \left(\frac{\left(\mathcal{K}_{1}\mathrm{e}^\frac{\mathcal{K}_{2}}{T}\right)^\frac{1}{\beta}(1-\beta)^\frac{1-\beta}{2\beta}}{\varepsilon^\frac{1-\beta}{\beta}}\right)^2\beta \left\Vert \upsilon \left( \cdot,T\right)\right\Vert _{L^2(\omega)}^2\\
&\quad\quad\quad\quad+\varepsilon^2\left\Vert \vartheta(\cdot,0)\right\Vert^2\text{.}
\end{array}
\]
Thus,
\begin{equation*}
\left\Vert \vartheta \left(\cdot,T\right)\right\Vert^2\leq
\left(\frac{\left(\mathcal{K}_{1}\mathrm{e}^\frac{\mathcal{K}_{2}}{T}\right)^\frac{1}{\beta}(1-\beta)^\frac{1-\beta}{2\beta}}{\varepsilon^\frac{1-\beta}{\beta}}\right)^2\beta \left\Vert \upsilon \left( \cdot,T\right)\right\Vert _{L^2(\omega)}^2+\varepsilon^2\left\Vert \vartheta(\cdot,0)\right\Vert^2\text{.}
\end{equation*}
Therefore, we obtain our desired estimate \eqref{2.38.} with
$$
\mathcal{M}_{1}:=\mathcal{K}_{1}^{\frac{1}{\beta}}(1-\beta)^{\frac{1-\beta}{2 \beta}} \beta^{\frac{1}{2}} ; \quad \mathcal{M}_{2}:=\frac{\mathcal{K}_{2}}{\beta} ;\quad \delta:=\frac{1-\beta}{\beta}.
$$
\end{proof}
\section{Approximate impulse controllability} \label{sec4}
Now, we study the impulse approximate controllability of system \eqref{1.1} by using Lemma \ref{lem1.1}.
\begin{defi}[see Definition 1.2 of \cite{QSGW}] 
System \eqref{1.1} is null approximate impulse controllable at time $T$ if for any $\varepsilon > 0$ and
any $\Psi^0=\left(\psi^{0},c,d\right) \in \mathbb{L}^2$, there exists a control function $h \in L^2(\omega),$ such that the associated state at final time satisfies
\begin{equation*}
\|\Psi(\cdot, T)\| \leq \varepsilon\left\Vert \Psi^0\right\Vert .
\end{equation*}
\end{defi}
If the system \eqref{1.1} is null approximate impulse controllable at time $T$, then for every $\varepsilon >0$ and $\Psi^0 \in \mathbb{L}^2$, the set
\begin{equation*}
\mathcal{C}_{T, \Psi^{0}, \varepsilon}:=\left\{h \in L^{2}(\omega): \text { the solution of }\eqref{1.1}\text { satisfies }\left\Vert\Psi(\cdot, T)\right\Vert \leq \varepsilon\left\Vert \Psi^{0}\right\Vert\right\}
\end{equation*}
is nonempty. In this case, we define the cost of null approximate impulse controllability as follows
$$K(T,\varepsilon):=\sup_{\left\|\Psi^0\right\|=1} \inf_{h \in \mathcal{C}_{T, \Psi^{0}, \varepsilon}} \|h\|_{L^{2}(\omega)}.
$$

\begin{rmq}
Since the semigroup $\left(\mathrm{e}^{t \mathbf{A}}\right)_{t\ge 0}$ is analytic, the range $\mathcal{R}\left(\mathrm{e}^{T \mathbf{A}}\right)$ is dense in $\mathbb{L}^2$. Hence, the null approximate impulse controllability is equivalent to the approximate impulse controllability:
$$\forall \Psi^0, \Psi^T \in \mathbb{L}^2, \;\forall \varepsilon>0,\; \exists\, h \in L^2(\omega) \colon \left\|\Psi(\cdot,T)-\Psi^T\right\| \le \varepsilon \|\Psi^0\|.$$
Thus, we will use ``approximate controllability" instead of ``null approximate controllability".
\end{rmq}
Next, we state the main result on approximate impulse controllability for system \eqref{1.1}.
\begin{teo}
The system \eqref{1.1} is approximate impulse controllable at any time $T > 0$. Moreover, for any $\varepsilon > 0,$ the cost of approximate impulse control satisfies
$$K(T, \varepsilon) \leq \frac{\mathcal{M}_{1} \mathrm{e}^{\frac{\mathcal{M}_{2}}{T-\tau}}}{\varepsilon^{\delta}},$$ where the positive constants $\mathcal{M}_1$, $\mathcal{M}_2$ and $\delta$ are from the estimate \eqref{2.38.}.
\end{teo}
\begin{proof}
Consider the following system
\begin{empheq}[left = \empheqlbrace]{alignat=2}\label{3.41}
\begin{aligned}
&\partial_{t} \upsilon-\partial_{xx} \upsilon=0, && \qquad \text { in } (a,b) \times(0, T), \\
&\partial_{t}\upsilon(a,t) - \partial_{x} \upsilon(a,t) =0, && \qquad \text { in } (0, T), \\
&\partial_{t}\upsilon(b,t) + \partial_{x} \upsilon(b,t) =0, && \qquad \text { in } (0, T), \\
&\left(\upsilon\left(\cdot,0\right),\upsilon\left(a,0\right),\upsilon\left(b,0\right)\right)=\left(\upsilon^0,c_2,d_2\right):=\vartheta^{0}, && \qquad\text { in } (a,b).\\
\end{aligned}
\end{empheq}
Let us fix $\varepsilon>0,$ $\Psi^{0}\in \mathbb{L}^2$, and put $
\kappa := \dfrac{\mathcal{M}_{1} \mathrm{e}^{\frac{\mathcal{M}_{2}}{T-\tau}}}{\varepsilon^{\delta}}$. We define the functional $J_{\varepsilon,\kappa}: \mathbb{L}^2 \rightarrow \mathbb{R}$ as follows
\begin{equation*}
J_{\varepsilon,\kappa}\left(\vartheta^{0}\right)=\frac{\kappa^{2}}{2}\|\upsilon(\cdot, T-\tau)\|_{L^{2}(\omega)}^{2}+\frac{\varepsilon^{2}}{2}\left\|\vartheta^{0}\right\|^2 + \left\langle \Psi^{0}, \vartheta(\cdot,T) \right\rangle,
\end{equation*}
where $\vartheta=(\upsilon,\upsilon(a,\cdot),\upsilon(b,\cdot))$ is the solution of \eqref{3.41}.
Notice that $J_{\varepsilon,\kappa}$ is strictly convex, $C^1$ and coercive, i.e.,  $J_{\varepsilon,\kappa}\left(\vartheta^{0}\right) \rightarrow \infty$ when $\left\|\vartheta^{0}\right\| \rightarrow \infty$. Therefore, $J_{\varepsilon,\kappa}$ has a unique minimizer $\tilde{\vartheta}^{0} \in \mathbb{L}^2$ such that 
$$J_{\varepsilon,\kappa}\left(\tilde{\vartheta}^{0}\right)=\min\limits_{\vartheta^{0} \in \mathbb{L}^2} J_{\varepsilon,\kappa}\left(\vartheta^{0}\right).$$ 
It implies that $J_{\varepsilon,\kappa}^{\prime}\left(\tilde{\vartheta}^{0}\right) \zeta^{0}=0$ for all $\zeta^{0} \in \mathbb{L}^2,$ i.e., the following estimate holds for any $\zeta^{0}$
\begin{equation}\label{4.39}
\kappa^{2} \int_{\omega} \tilde{\upsilon}(x, T-\tau) z(x, T-\tau) \mathrm{d} x+\varepsilon^{2} \left\langle \tilde{\vartheta}^{0},\zeta^{0} \right\rangle+\left\langle \Psi^{0}, \zeta(\cdot,T) \right\rangle = 0,
\end{equation}
where $\tilde{\vartheta}=\left(\tilde{\upsilon},\tilde{\upsilon}(a,\cdot),\tilde{\upsilon}(b,\cdot)\right)$ and $\zeta=(z,z(a,\cdot),z(b,\cdot))$ are respectively the solutions of \eqref{3.41} corresponding to $\tilde{\vartheta}^{0}$ and $\zeta^{0}$.
Recall that $\psi$ satisfies
\begin{empheq}[left = \empheqlbrace]{alignat=2} \label{5.1}
\begin{aligned}
&\partial_{t} \psi(x,t)-\partial_{xx} \psi(x,t)=0, && \qquad (x,t)\in(a,b) \times(0, T) \backslash\{\tau\},\\
&\psi(x, \tau)=\psi\left(x, \tau^{-}\right)+\mathds{1}_{\omega}(x) h(x), && \qquad x\in (a,b),\\
&\partial_{t}\psi(a,t) - \partial_{x} \psi(a,t)=0, && \qquad t\in(0, T) \backslash\{\tau\}, \\
&\partial_{t}\psi(b,t) + \partial_{x} \psi(b,t)=0, && \qquad t\in(0, T) \backslash\{\tau\}, \\
&\psi(a, \tau)=\psi\left(a, \tau^{-}\right),\; \psi(b, \tau)=\psi\left(b, \tau^{-}\right),\\
& \left(\psi(x,0),\psi(a, 0),\psi(b, 0)\right)=\left(\psi^0(x),c,d\right), && \qquad x\in(a,b),
\end{aligned}
\end{empheq}
Multiplying \eqref{5.1}$_{1}$ by $z(\cdot,T-t)$, \eqref{5.1}$_{3}$ by $z(a,T-t)$ and \eqref{5.1}$_{4}$ by $z(b,T-t)$ for all $t\in (0,T) \backslash\{\tau\}$, we obtain
\[
\begin{array}{ll}
&\displaystyle\int_{a}^b\partial_{t}\psi(x,t)z(x,T-t) \mathrm{d}x - \int_{a}^b\partial_{xx}\psi(x,t)z(x,T-t) \mathrm{d}x 
+\partial_{t}\psi(a,t)z(a,T-t) \\
&-\partial_{x}\psi(a,t)z(a,T-t)+\partial_{t}\psi(b,t)z(b,T-t)+\partial_{x}\psi(b,t)z(b,T-t)  =0.
\end{array}
\]
Integrating by parts twice, we obtain
\[
\begin{array}[c]{ll}
&\displaystyle\int_{a}^b\partial_{t}\psi(x,t)z(x,T-t) \mathrm{d}x+\psi(b,t)\partial_{x}z(b,T-t)-\psi(a,t)\partial_{x}z(a,T-t) \\
&-\displaystyle\int_{a}^b\psi(x,t)\partial_{xx} z(x,T-t) \mathrm{d}x +\partial_{t}\psi(a,t)z(a,T-t)+\partial_{t}\psi(b,t)z(b,T-t)=0.
\end{array}
\]
Since $\zeta=(z,z(a,\cdot),z(b,\cdot))$ is the solution of \eqref{3.41}, then $\partial_{xx} z(\cdot,T-t)=\partial_{t}z(\cdot,T-t)$, $$\partial_{x}z(a,T-t)=\partial_{t}z(a,T-t)\; \text{and}\; \partial_{x}z(b,T-t)=-\partial_{t}z(b,T-t).$$
Therefore,
\[
\begin{array}[c]{ll}
&\displaystyle\int_{a}^b\partial_{t}\psi(x,t)z(x,T-t) \mathrm{d}x-\int_{a}^b\psi(x,t)\partial_{t} z(x,T-t) \mathrm{d}x  -\psi(b,t)\partial_{t}z(b,T-t)\\
&-\psi(a,t)\partial_{t}z(a,T-t)+\partial_{t}\psi(a,t)z(a,T-t)+\partial_{t}\psi(b,t)z(b,T-t)=0.
\end{array}
\]
That is, 
\begin{equation}\label{3.44}
  \displaystyle\int_{a}^b\frac{\mathrm{d}}{\mathrm{d}\mathrm{t}}\bigg( \psi(x,t)z(x,T-t) \bigg) \mathrm{d}x + \frac{\mathrm{d}}{\mathrm{d} \mathrm{t}} \bigg( \psi(a,t)z(a,T-t)\bigg) + \frac{\mathrm{d}}{\mathrm{d} \mathrm{t}} \bigg( \psi(b,t)z(b,T-t)\bigg)  = 0.
\end{equation}
Integrating \eqref{3.44} over $(0, \tau)$ yields
\begin{equation*}
\int_{a}^b\left[\psi(x,t)z(x,T-t)\right]_{0}^{\tau} \mathrm{d} x+\left[\psi(a,t)z(a,T-t)\right]_{0}^{\tau}+\left[\psi(b,t)z(b,T-t)\right]_{0}^{\tau} =0.
\end{equation*}
Therefore,
\begin{align}\label{4.41}
&\int_{a}^b \left(\psi(x,\tau^{-})z(x,T-\tau) -\psi(x,0)z(x,T)\right) \mathrm{d}x+\psi(a,\tau^{-})z(a,T-\tau) -\psi(a,0)z(a,T)\nonumber\\
&+\psi(b,\tau^{-})z(b,T-\tau) -\psi(b,0)z(b,T)=0.
\end{align}
Integrating \eqref{3.44} over $(\tau, T)$ yields
\begin{equation*}
\int_{a}^b\left[\psi(x,t)z(x,T-t)\right]_{\tau}^{T} \mathrm{d}x+\left[\psi(a,t)z(a,T-t)\right]_{\tau}^{T}+\left[\psi(b,t)z(b,T-t)\right]_{\tau}^{T}=0.
\end{equation*}
Hence
\begin{align}\label{4.42}
&\int_{a}^b\left(\psi(x,T)z(x,0) -\psi(x,\tau)z(x,T-\tau)\right) \mathrm{d}x+\psi(a,T)z(a,0) -\psi(a,\tau)z(a,T-\tau)\nonumber \\
&+\psi(b,T)z(b,0) -\psi(b,\tau)z(b,T-\tau)=0.
\end{align}
Combining \eqref{4.41}-\eqref{4.42} and using the fact that $\psi(\cdot, \tau)=\psi\left(\cdot, \tau^{-}\right)+\mathds{1}_{\omega} h(\cdot)$, $\psi(a, \tau)=\psi\left(a, \tau^{-}\right)$ and $\psi(b, \tau)=\psi\left(b, \tau^{-}\right)$, we obtain
\begin{equation}\label{4.43}
\begin{aligned}
&\int_{\omega} h(x)z(x,T-\tau) \mathrm{d} x+\int_{a}^b\psi(x,0)z(x,T)\mathrm{d} x-\int_{a}^b \psi(x,T)z(x,0) \mathrm{d}x+\psi(a,0)z(a,T)  \\
& -\psi(a,T)z(a,0)+\psi(b,0)z(b,T)-\psi(b,T)z(b,0) = 0.
\end{aligned}
\end{equation}
Thus, if we choose $h(x)=\kappa^{2} \tilde{v}(x, T-\tau)$, we obtain, from \eqref{4.39} and \eqref{4.43}, that
\begin{equation*}
\left\langle \Psi(\cdot, T) +\varepsilon^{2} \tilde{\vartheta}^{0}, \zeta^{0} \right\rangle =0, \qquad \forall \zeta^{0} \in \mathbb{L}^2.
\end{equation*}
Hence, $\Psi(x, T) = -\varepsilon^{2} \tilde{\vartheta}^{0}(x).$ Moreover, with $\zeta^{0} =\tilde{\vartheta}^{0},$ using the Cauchy-Schwarz inequality, it follows from \eqref{4.39} that
\begin{equation*}
\kappa^{2}\|\tilde{\upsilon}(\cdot, T-\tau)\|_{L^{2}(\omega)}^{2}+\varepsilon^{2}\left\Vert\tilde{\vartheta}^{0}\right\Vert ^2 \leq\left\|\Psi^{0}\right\| \|\tilde{\vartheta}(\cdot, T)\|.
\end{equation*}
By virtue of the energy estimate for the system \eqref{3.41}, which is
\begin{equation*}
\|\tilde{\vartheta}(\cdot, T)\| \leq\|\tilde{\vartheta}(\cdot, T-\tau)\|,
\end{equation*}
we obtain
\begin{equation}\label{3.51}
\kappa^{2}\|\tilde{\upsilon}(\cdot, T-\tau)\|_{L^{2}(\omega)}^{2}+\varepsilon^{2}\left\|\tilde{\vartheta}^{0}\right\| ^2\leq\left\|\Psi^{0}\right\| \|\tilde{\vartheta}(\cdot, T-\tau)\|.
\end{equation} 
Applying the estimate of Lemma \ref{lemma3.1}, which is
\begin{equation*}
\|\tilde{\vartheta}(\cdot, T-\tau)\|^{2} \leq \kappa^{2}\|\tilde{\upsilon}(\cdot, T-\tau)\|_{L^{2}(\omega)}^{2}+\varepsilon^{2}\|\tilde{\vartheta}(\cdot, 0)\|^{2},
\end{equation*}
and using \eqref{3.51}, we obtain
\begin{equation}\label{3.53}
\|\tilde{\vartheta}(\cdot, T-\tau)\| \leq\left\|\Psi^{0}\right\|.
\end{equation}
Finally, combining \eqref{3.51} and \eqref{3.53}, we obtain
\begin{equation*}
\kappa^{2}\|\tilde{\upsilon}(\cdot, T-\tau)\|_{L^{2}(\omega)}^{2}+\varepsilon^{2}\left\|\tilde{\vartheta}^{0}\right\| ^2\leq\left\|\Psi^{0}\right\|^2.
\end{equation*}
Recall that $\Psi(x, T) = -\varepsilon^{2} \tilde{\vartheta}^{0}(x)$ and $h(x)=\kappa^{2} \tilde{v}(x, T-\tau)$. Thus
\begin{equation*}
\frac{1}{\kappa^{2}}\|h\|_{L^{2}(\omega)}^{2}+\frac{1}{\varepsilon^{2}}\|\Psi(\cdot, T)\|^{2} \leq\left\|\Psi^{0}\right\|^{2}.
\end{equation*}
This completes the proof.
\end{proof}

\section{An algorithm for computing HUM impulse controls} \label{sec5}
In this section, we present a numerical method to compute the HUM impulse controls. This will be done based on a penalized HUM approach combined with a CG algorithm.
\subsection{The HUM impulse controls}

Let $\Psi^0$ be an initial datum to be controlled (for notational simplicity we assume that $\|\Psi^0\|=1$) and let $\varepsilon>0$ be fixed. We define the cost functional $J_{\varepsilon}: \mathbb{L}^2 \rightarrow \mathbb{R}$ by
\begin{equation*}
J_{\varepsilon}\left(\vartheta^{0}\right)=\frac{1}{2}\|\upsilon(\cdot, T-\tau)\|_{L^{2}(\omega)}^{2}+\frac{\varepsilon}{2}\left\|\vartheta^{0}\right\|^2 + \left\langle \Psi^{0}, \vartheta(\cdot,T) \right\rangle,
\end{equation*}
where $\vartheta=(\upsilon,\upsilon(a,\cdot),\upsilon(b,\cdot))$ is the solution of \eqref{3.41}. The unique minimizer $\tilde{\vartheta}^{0}_\varepsilon \in \mathbb{L}^2$ of $J_{\varepsilon}$ is characterized by the Euler-Lagrange equation
\begin{equation}\label{Eq2}
\int_{\omega} \tilde{\upsilon}_\varepsilon(x, T-\tau) z(x, T-\tau) \mathrm{d} x+\varepsilon \left\langle \tilde{\vartheta}^{0}_\varepsilon,\zeta^{0} \right\rangle+\left\langle \Psi^{0}, \zeta(\cdot,T) \right\rangle = 0
\end{equation}
for all $\zeta^0 \in \mathbb{L}^2$, where $\tilde{\vartheta}_\varepsilon$ and $\zeta=(z,z(a,\cdot),z(b,\cdot))$ are respectively the solutions of \eqref{3.41} associated to $\tilde{\vartheta}^{0}_\varepsilon$ and $\zeta^{0}$. Let us introduce the control operator $\mathcal{B}\colon \mathbb{L}^2 \rightarrow \mathbb{L}^2$ defined by
$$\mathcal{B}(\upsilon, c,d)=(\mathds{1}_\omega\upsilon,0,0)$$
and the non-negative symmetric operator
$$\Lambda_\tau \colon \mathbb{L}^2 \rightarrow \mathbb{L}^2,$$
usually called the Gramian operator, defined by
$$\Lambda_\tau \varrho=\mathrm{e}^{(T-\tau)\mathbf{A}} \mathcal{B} \,\mathrm{e}^{(T-\tau)\mathbf{A}} \varrho.$$
Then, the impulse HUM control is given by
$$\widehat{h}=\mathcal{B} \,\mathrm{e}^{(T-\tau)\mathbf{A}} \tilde{\vartheta}^{0}_\varepsilon,$$
and the equation \eqref{Eq2} can be reformulated as
\begin{equation}
    \left(\Lambda_\tau +\varepsilon \mathbf{I}\right) \tilde{\vartheta}^{0}_\varepsilon = -\mathrm{e}^{T\mathbf{A}} \Psi^0.
\end{equation}
To solve this last problem, we will use the following CG algorithm.

\begin{algorithm}[H]
\noindent Set $k=0$ and choose an initial guess $\mathbf{f}_0=\left(f_0,f_0^a,f_0^b\right) \in \mathbb{L}^2$.\;
Solve the problem
\begin{empheq}[left = \empheqlbrace]{alignat=2}\label{1.3}
\begin{aligned}
&\partial_{t} p_0(x,t)-\partial_{xx} p_0(x,t)=0, && \quad (x,t)\in (a,b) \times(0, T) , \\
&\partial_{t}p_0(a,t) - \partial_{x}p_0(a,t) =0, && \quad t\in (0, T), \\
&\partial_{t}p_0(b,t) + \partial_{x}p_0(b,t) =0, && \quad t\in (0, T), \\
&\left(p_0(x,0),p_0(a, 0),p_0(b, 0)\right)=\mathbf{f}_0(x), && \quad x\in (a,b), \notag
\end{aligned}
\end{empheq}
and set $\mathbf{u}_0(x)=\mathcal{B}\, \mathbf{p}_0(T-\tau,x)$ \;
Solve the problem
\begin{empheq}[left = \empheqlbrace]{alignat=2}\label{1.3}
\begin{aligned}
&\partial_{t} y_0(x,t)-\partial_{xx} y_0(x,t)=0, && \, (x,t)\in (a,b) \times(0, T) , \\
&\partial_{t}y_0(a,t) - \partial_{x}y_0(a,t) =0, && \, t\in (0, T), \\
&\partial_{t}y_0(b,t) + \partial_{x}y_0(b,t) =0, && \, t\in (0, T), \\
&\left(y_0(x,0),y_0(a, 0),y_0(b, 0)\right)=\mathbf{u}_0(x), && \, x\in (a,b), \notag
\end{aligned}
\end{empheq}
compute $\mathbf{g}_0=\varepsilon \mathbf{f}_0 + Y_0(T-\tau) + \mathrm{e}^{T\mathbf{A}} \Psi^0$ and set $\mathbf{w}_0=\mathbf{g}_0$
\;
For $k=1,2,\ldots,$ until convergence,
solve the problem
\begin{empheq}[left = \empheqlbrace]{alignat=2}\label{1.3}
\begin{aligned}
&\partial_{t} p_k(x,t)-\partial_{xx} p_k(x,t)=0, && \, (x,t)\in (a,b) \times(0, T) , \\
&\partial_{t}p_k(a,t) - \partial_{x}p_k(a,t) =0, && \, t\in (0, T), \\
&\partial_{t}p_k(b,t) + \partial_{x}p_k(b,t) =0, && \, t\in (0, T), \\
&\left(p_k(x,0),p_k(a, 0),p_k(b, 0)\right)=\mathbf{w}_{k-1}(x), && \, x\in (a,b), \notag
\end{aligned}
\end{empheq}
and set $\mathbf{u}_k(x)=\mathcal{B}\, \mathbf{p}_k(T-\tau,x)$\;
Solve the problem
\begin{empheq}[left = \empheqlbrace]{alignat=2}\label{1.3}
\begin{aligned}
&\partial_{t} y_k(x,t)-\partial_{xx} y_k(x,t)=0, && \, (x,t)\in (a,b) \times(0, T) , \\
&\partial_{t}y_k(a,t) - \partial_{x}y_k(a,t) =0, && \, t\in (0, T), \\
&\partial_{t}y_k(b,t) + \partial_{x}y_k(b,t) =0, && \, t\in (0, T), \\ 
&\left(y_k(x,0),y_k(a, 0),y_k(b, 0)\right)=\mathbf{u}_k(x), && \, x\in (a,b), \notag
\end{aligned}
\end{empheq}
and compute $$\bar{\mathbf{g}}_k=\varepsilon \mathbf{w}_{k-1} + Y_k(T-\tau) \qquad \text{ and } \qquad \rho_k=\dfrac{\|\mathbf{g}_{k-1}\|^2}{\langle \bar{\mathbf{g}}_k, \mathbf{w}_{k-1} \rangle},$$
then
$$\mathbf{f}_k=\mathbf{f}_{k-1}-\rho_k \mathbf{w}_{k-1} \qquad \text{ and } \qquad \mathbf{g}_k=\mathbf{g}_{k-1}-\rho_k \bar{\mathbf{g}}_k$$\;
\noindent \textbf{If} $\dfrac{\|\mathbf{g}_{k}\|}{\|\mathbf{g}_{0}\|} \le tol$, stop the algorithm, set $\mathbf{g}=\mathbf{f}_{k}$ and solve the problem
\begin{empheq}[left = \empheqlbrace]{alignat=2}\label{1.3}
\begin{aligned}
&\partial_{t} p_k(x,t)-\partial_{xx} p_k(x,t)=0, && \, (x,t)\in (a,b) \times(0, T) , \\
&\partial_{t}p_k(a,t) - \partial_{x}p_k(a,t) =0, && \, t\in (0, T), \\
&\partial_{t}p_k(b,t) + \partial_{x}p_k(b,t) =0, && \, t\in (0, T), \\
&\left(p_k(x,0),p_k(a, 0),p_k(b, 0)\right)=\mathbf{g}(x), && \, x\in (a,b), \notag
\end{aligned}
\end{empheq}
and set $\mathbf{u}_k(x)=\mathcal{B}\, \mathbf{p}_k(T-\tau,x)$\;
\textbf{Else} compute $$\gamma_k=\dfrac{\|\mathbf{g}_{k}\|^2}{\|\mathbf{g}_{k-1}\|^2} \qquad \text{ and then } \qquad \mathbf{w}_k=\mathbf{g}_{k}+\gamma_k \mathbf{w}_{k-1}.$$
\caption{HUM with CG Algorithm}
\label{alg1}
\end{algorithm}

\newpage
\subsection{Numerical experiments}
Next, we present some numerical tests to illustrate our theoretical
results and show the efficiency of the CG algorithm presented above.

We use the method of lines to numerically solve different PDEs. A uniform space grid $x_i=i \Delta x$ for  $i=\overline{1,N_x}$, with $\Delta x=\dfrac{1}{N_x}$, is used to divide the space interval $[a,b]$ for the numerical resolution of systems with dynamic boundary conditions, used in Algorithm \ref{alg1}. We take $N_x=25$ as a space mesh parameter.

In the numerical tests, we will choose the following values
$$T=0.02, \quad \tau=0.01, \quad a=0,\; b=1, \quad \omega=(0.2, 0.8) \Subset (0,1),$$
and we consider the initial datum to be controlled as
$$\psi_0(x)=\sqrt{2} \sin(\pi x), \qquad x \in [0,1].$$
Next, we plot the uncontrolled and the controlled solutions.
\begin{figure}[H]
\centering
\includegraphics[scale=0.5]{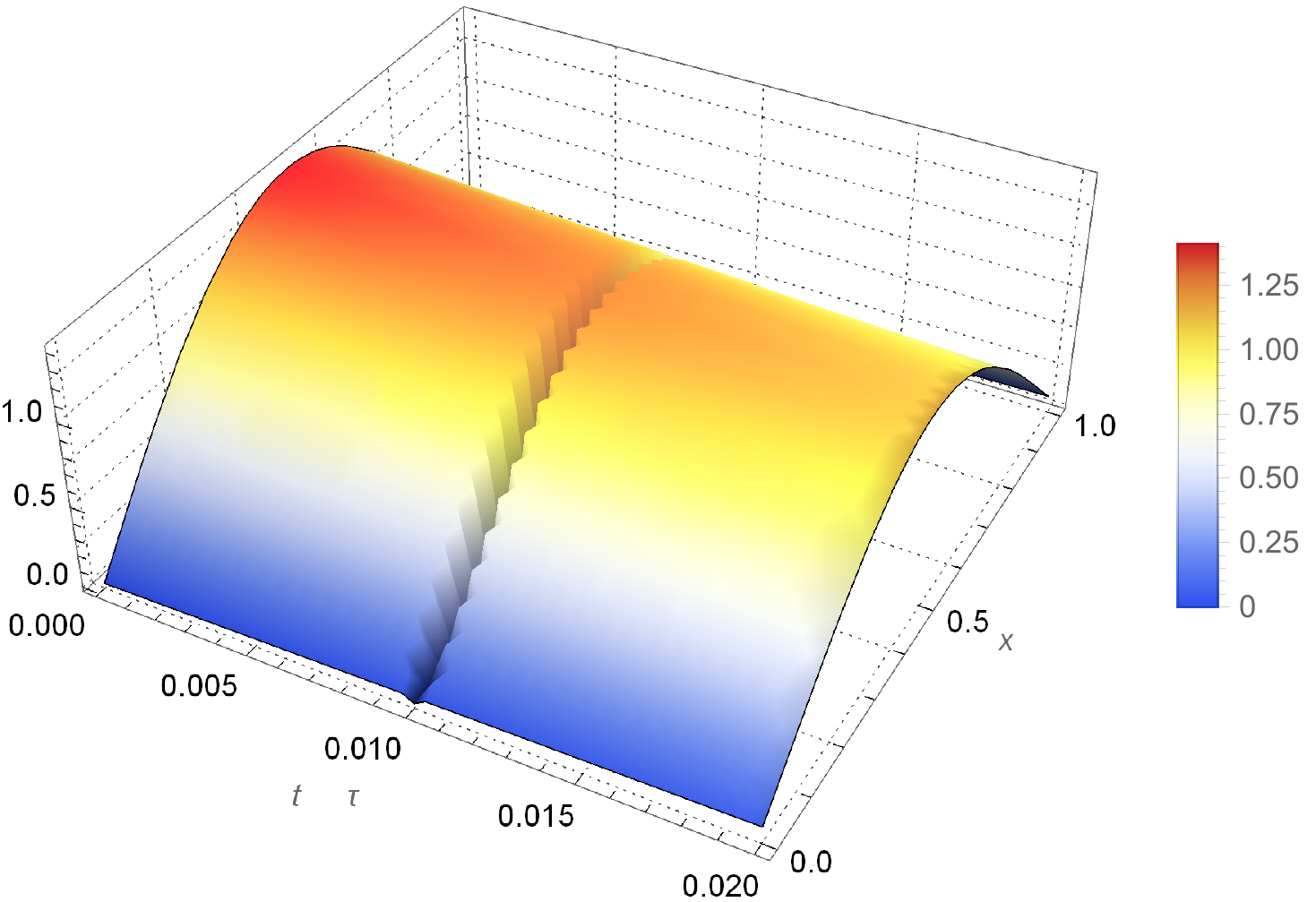}
\caption{The uncontrolled solution.}
\end{figure}

\begin{figure}[H]
\centering
\includegraphics[scale=0.6]{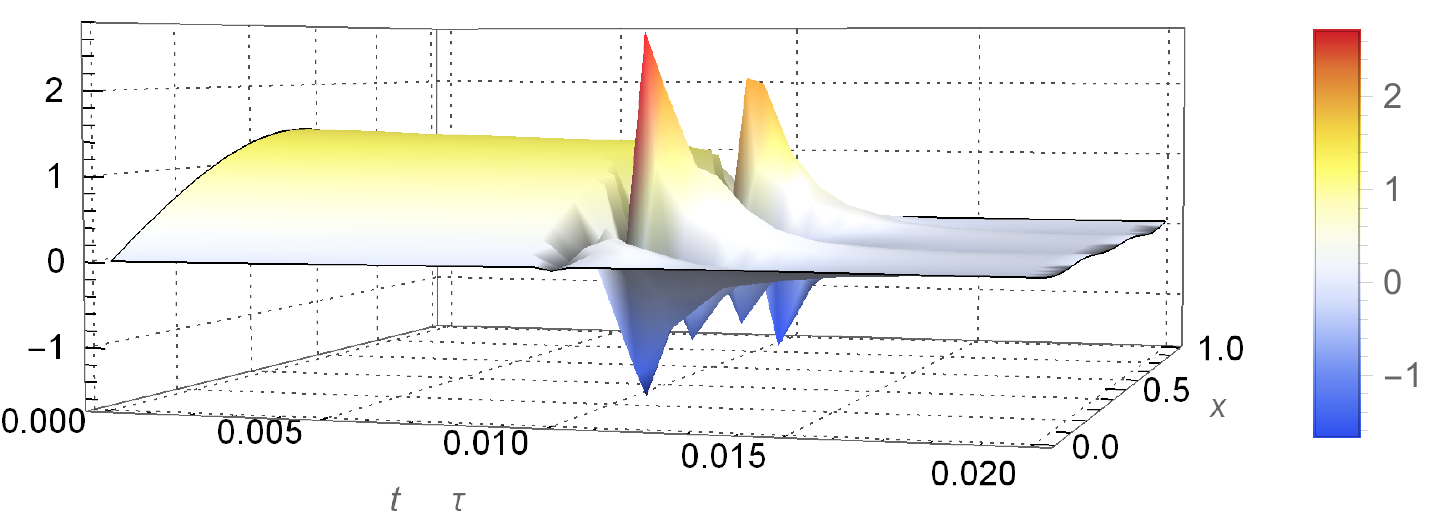}
\caption{The controlled solution.}
\label{fig1}
\end{figure}

The initial iteration of the algorithm is chosen as $\mathbf{f}_0=0$. We have chosen $\varepsilon=10^{-4}$ and the stopping parameter $tol=10^{-3}$. The algorithm stops at the iteration number $k_*=25$. In Figure \ref{fig1}, we clearly see the effect of the impulse control at time $\tau=0.01$.

\begin{figure}[H]
\centering
\includegraphics[scale=0.5]{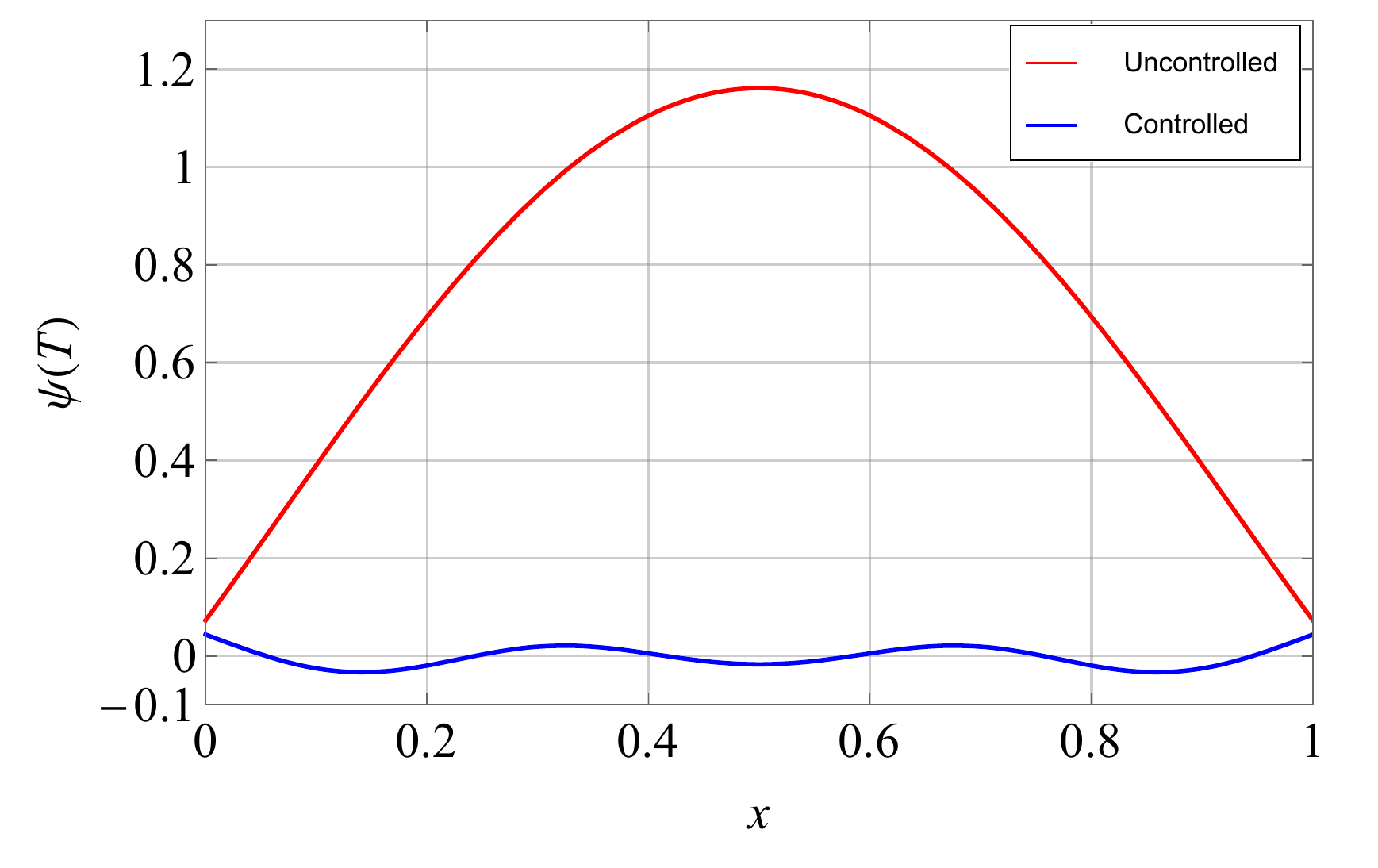}
\caption{The final state for uncontrolled and controlled solutions.}
\end{figure}
\bigskip

\begin{figure}[H]
\centering
\includegraphics[scale=0.5]{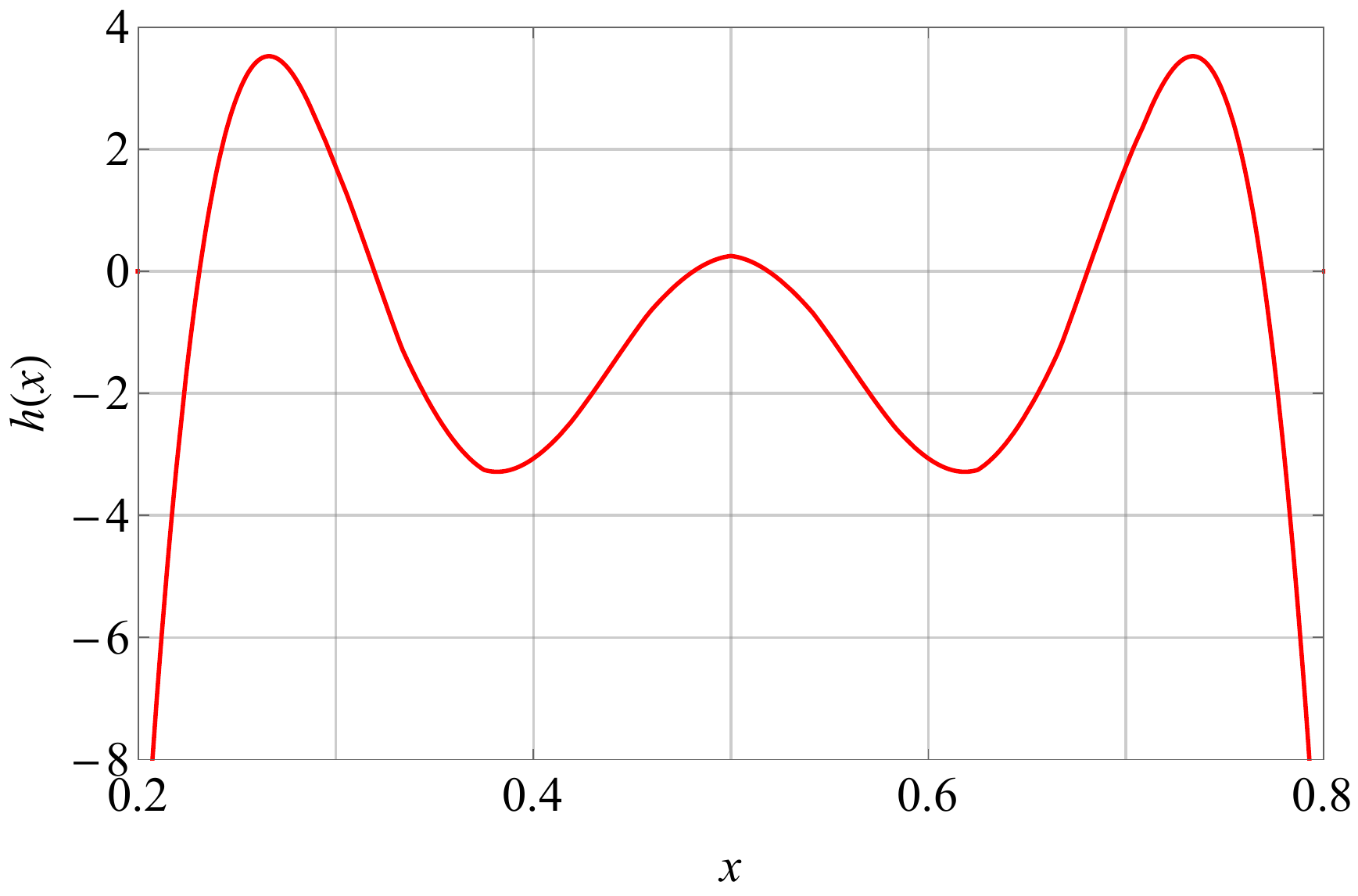}
\caption{The computed control $h$.}
\end{figure}
\bigskip

\begin{table}[ht]
\caption{Numerical results for $T=0.02$, $\tau=0.01$ and $tol=10^{-3}$.}
\centering\setlength\arraycolsep{13pt}
$\begin{array}{cccc}
\hline\\[-3mm]
 \varepsilon  & 10^{-2} &  10^{-3} & 10^{-4} \\
\specialrule{.1em}{.1em}{.1em}\\[-3mm]
 N_{\text{iter}} & 6 & 10 & 25 \\
\hline\\[-3mm]
 \|\Psi(T)\| & 8.54\times 10^{-2} & 7.27\times 10^{-2} & 6.47\times 10^{-2} \\
\hline\\[-3mm]
 \|h\|_{L^2(\omega)} & 0.9478 & 1.1325 & 2.2109 \\
\hline\\[-3mm]
\end{array}$
\end{table}
We clearly see that the distance $\|\Psi(T)\|$ to the target zero decreases and the norm of the impulse control $\|h\|_{L^2(\omega)}$ increases as $\varepsilon$ diminishes.

\begin{rmq}
The above numerical tests show that the proposed algorithm yields accurate and fast results for the numerical computation of distributed impulsive controls controlling the heat equation with dynamic boundary conditions at a single instant of time $\tau$. This approach can be generalized for more general and multi-dimensional parabolic systems with dynamic boundary conditions.
\end{rmq}

\section{Conclusions and Remarks} \label{sec6}
In this work, a logarithmic convexity result has been proved for the 1-D heat equation with dynamic boundary conditions. As an application, the impulsive approximate controllability for the system \eqref{1.1} has been established with an explicit bound of the cost. The proof is based on the Carleman commutator approach. Afterward, a constructive algorithm has been developed to numerically construct the impulse control of minimal $L^2$-norm. This has been done by combining a penalized HUM approach and a CG method. Finally, a numerical simulation has been performed to illustrate the theoretical result of impulse approximate controllability.

To the best of the authors knowledge, dynamical systems with impulsive controls have not been much studied numerically, which opens the doors to many possibilities for dealing with such problems. This work can be generalized in several ways, for instance, one would study the case of an infinite number of impulses or the case of some perturbations as: nonlinearities, delays and non local conditions. One would also change the type of impulses by considering non-instantaneous impulses. Such problems would be of much interest to investigate.

\section*{Acknowledgment}
The authors would like to express their thanks to anonymous referees for constructive comments and suggestions that improved the quality of this manuscript.

\end{document}